\documentclass[preprint]{elsarticle}
\usepackage{amssymb}
\usepackage{amsmath}
\usepackage{amssymb,amsmath,mathrsfs,bm,xcolor}
\usepackage{bbm}
\usepackage{cases}
\usepackage{array}
\usepackage{wasysym,latexsym}
\usepackage{graphicx,color,algorithm,algpseudocode}
\usepackage{booktabs,makecell,multirow}
\usepackage{epstopdf,epsfig}
\usepackage{diagbox}
\usepackage{float}
\usepackage{mathtools}
\usepackage{setspace}
\usepackage{amsthm}

\newtheorem{thm}{Theorem}
\newtheorem{lemma}{Lemma}
\newtheorem{remark}{Remark}

\newtheorem{assumption}{Assumption}
\usepackage{setspace}
\usepackage{graphicx} 
\usepackage{geometry}
\usepackage{multirow}
\usepackage{booktabs}
\usepackage{array}
\usepackage{xurl}
\usepackage{ulem}

\usepackage{algorithm,algpseudocode}
\usepackage{algpseudocode} 

\usepackage{caption} 
\usepackage{subcaption} 
\usepackage{adjustbox} 
\usepackage[absolute,overlay]{textpos} 
 \usepackage{geometry} 

\usepackage[colorlinks,linkcolor=blue,anchorcolor=blue,linktocpage=true,urlcolor=blue,citecolor=blue]{hyperref}

\usepackage{enumitem} 
\usepackage{hyperref} 

\newcommand{\msc}[1]{\textbf{MSC Classification:}#1}
\journal{ }
\graphicspath{{figures/}}
\begin{document}
	\begin{frontmatter}
    
    \title{A New Primal‑Dual Algorithm with Two Convex Combinations and Linesearch for General Convex‑Concave Saddle‑Point Problems}
    
    \author[1]{Zexian Liu}
    \ead{liuzexian2008@163.com}
    \author[1]{Shuning Liu}
	\ead{snliu00@163.com}
    \author[1]{Jialong Li \corref{cor1}}
	\ead{ ailijia0417@163.com}

    \address[1]{School of Mathematics and Statistics, Guizhou University, Guiyang 550025, China}
    \cortext[cor1]{Corresponding author}

    \begin{abstract}
        Convex-concave saddle-point problems  are ubiquitous across diverse domains, including machine learning, image processing, economics, and equilibrium problems. Primal-dual algorithms provide a highly effective and  powerful  framework for convex-concave saddle-point problems. Convex combination has become a crucial acceleration technique for primal-dual algorithms, and the integration of this technique has recently made these algorithms a highly active research topic. The choice of the convex combination parameter often has a significant impact on both the theoretical analysis and the numerical performance of the corresponding algorithms. However, the requirements  on this parameter imposed by theory are sometimes inconsistent with those suggested by numerical experiments. For instance, theoretical analysis often requires the parameter to be small, while numerical experiments tend sometimes to favor larger values. To address this inconsistency and further advance primal-dual algorithms with convex combination, we develop a novel strategy based on two convex combinations, integrate it into a primal-dual framework, and propose a new primal-dual algorithm with linesearch, termed NPDAL-n, for general convex-concave saddle-point problems. The proposed two convex combinations in NPDAL-n ensure that the permissible range of the convex combination parameters is mainly determined by theoretical considerations, with little regard for numerical performance. Through rigorous Lyapunov energy descent analysis, we establish the global convergence and a sublinear ergodic convergence rate of $\mathcal{O}(1/N)$ for NPDAL-n under standard assumptions. When the primal function is strongly convex, we develop an accelerated version of NPDAL-n that achieves an optimal $\mathcal{O}(1/N^2)$ rate. Furthermore, by analyzing the linesearch condition via the problem structure, we present an adaptive variant of NPDAL-n—a linesearch-free proximal gradient method for composite convex optimization problems, which can be regarded as a special case of convex-concave saddle-point problems. Comprehensive numerical experiments on quadratically constrained quadratic programming  and sparse logistic regression problems demonstrate that the proposed algorithms outperform some state-of-the-art algorithms.
    \end{abstract}		
	\begin{keyword}
        primal-dual algorithm \sep linesearch \sep convex combination \sep global convergence \sep convergence rate
		\end{keyword}
    \end{frontmatter}
    \msc{: 49M29, 65K10, 65Y20, 90C25}
    
\section{Introduction}\label{section1}
    In this paper, we consider the following general convex-concave saddle-point problem 
    \begin{equation}\label{p}
        \min_{x\in \mathbb{R}^{q}} \max_{y\in \mathbb{R}^{p}}\; L(x,y) = g(x) + \Phi(x, y) - f^*(y).
    \end{equation}
   where $g: \mathbb{R}^{q} \to (-\infty,+\infty]$ and $f: \mathbb{R}^{p} \to (-\infty,+\infty]$ are proper, closed, and convex extended real-valued functions, and $\Phi: \mathrm{dom}(g) \times \mathrm{dom}(f^{*}) \to \mathbb{R}$ is differentiable and is convex in  $x$ for each fixed $y$ and concave in $y$ for each fixed $x$. Here, $f^*$ denotes the Legendre-Fenchel conjugate of $f$, i.e., $f^{*}(y) = \sup_{u \in \mathbb{R}^{p}} \langle y, u \rangle - f(u), \ \forall\, y \in \mathbb{R}^{p}$.  
    
    The generic formulation in \eqref{p} is highly versatile, capturing a wide spectrum of fundamental models explored in the optimization literature. For example, replacing $\Phi(x,y)$ with a standard bilinear form $\langle Kx, y \rangle$ (where $K \in \mathbb{R}^{p \times q}$) directly yields the classical   bilinear saddle-point problem. This specific formulation serves as a cornerstone model extensively utilized in signal and image processing, machine learning, statistics, economics, and mechanics \cite{ref1,ref2,ref3,ref4,ref5,ref11}.

    By specializing the functions  $g$ and $f^*$, problem \eqref{p} recovers several other prominent paradigms. For instance, setting them as indicator functions over unit simplices reduces the framework to a zero-sum matrix game. Moreover, when $g$ and $f^*$   are null functions, the problem simplifies to an unconstrained saddle‑point configuration \cite{ref7, ref8}, which has gained increasing relevance in modern machine learning applications \cite{ref6, ref9}. More broadly, problem \eqref{p} naturally encompasses convex optimization scenarios subject to nonlinear conic constraints, thereby seamlessly unifying linear, quadratic, quadratically constrained quadratic, second‑order cone, and semidefinite programming under a single umbrella \cite{ref10}. Furthermore, this framework elegantly accommodates composite convex optimization problems featuring a finite-sum structure:
    \begin{equation}\label{finite_sum}
        \min_{x\in \mathbb{R}^{q}} \; g(x) + \frac{1}{p}\sum_{i=1}^{p} h_i(x).
    \end{equation}
    Such finite-sum models are ubiquitous in empirical risk minimization and large-scale data analysis. By defining the mapping $H(x) := (h_1(x), \dots, h_p(x))^\top$ and introducing a dual variable $y \in \mathbb{R}^p$, problem \eqref{finite_sum} can be   reformulated as 
    \begin{equation}\label{finite_sum_saddle}
        \min_{x\in \mathbb{R}^{q}} \max_{y\in \mathbb{R}^{p}} \; g(x) + \langle H(x), y \rangle - \iota_{\{\mathbf{1}/p\}}(y),
    \end{equation}
    where $\mathbf{1} \in \mathbb{R}^p$ denotes the vector of all ones, and $\iota_{\{\mathbf{1}/p\}}$ is the indicator function of the singleton set $\{\mathbf{1}/p\}$. This explicit representation demonstrates that the finite-sum problem naturally fits into the saddle-point problem \eqref{p} with $\Phi(x,y) = \langle H(x), y \rangle$. 
    
    Given the pervasive applications of problem \eqref{p}, a myriad of numerical strategies have been developed, prominently including the extragradient \cite{ref12, ref13} and optimistic gradient descent-ascent schemes \cite{ref7, ref8}. From a historical perspective, primal-dual algorithms (PDAs)---tracing back to the foundational work of Arrow and Hurwicz \cite{ref18}---have evolved significantly over the decades to tackle such saddle-point problems \cite{ref3, ref15, ref10, ref16, ref17}. 

    \subsection{Notation}\label{notation}
    Throughout this paper, we adopt standard notations from convex analysis. For a finite-dimensional Euclidean space $\mathbb{R}^m$, the standard inner product and its induced Euclidean norm are denoted by $\langle \cdot, \cdot \rangle$ and $\|\cdot\|=\sqrt{\langle\cdot,\cdot\rangle}$, respectively. Let $\phi: \mathbb{R}^m \rightarrow (-\infty, +\infty]$ be a proper, lower semicontinuous, and convex function. The effective domain of $\phi$ is defined as $\text{dom}(\phi) := \{x \in \mathbb{R}^m \mid \phi(x) < +\infty\}$. 
    
    If $\phi$ is differentiable, $\nabla \phi(x)$ represents its gradient; otherwise, $\partial \phi(x) := \{\xi \in \mathbb{R}^m \mid \phi(y) \ge \phi(x) + \langle \xi, y - x \rangle, \forall y \in \text{dom}(\phi)\}$ denotes its subdifferential at $x$. We say that $\nabla \phi$ is $L_\phi$-smooth if it is Lipschitz continuous on the interior of $\text{dom}(\phi)$ with a constant $L_\phi \ge 0$, satisfying $\|\nabla \phi(x) - \nabla \phi(y)\| \le L_\phi \|x - y\|$. 
    
    For any scalar $\lambda > 0$, the proximal operator associated with the scaled function $\lambda \phi$ is uniquely defined as 
    $$
    \text{Prox}_{\lambda \phi}(x) := \operatorname*{argmin}_{y \in \mathbb{R}^m} \left\{ \phi(y) + \frac{1}{2\lambda} \|y - x\|^2 \right\}.
    $$
    Finally, for column vectors $u$ and $v$, their concatenation is compactly denoted by $(u; v) := (u^\top, v^\top)^\top$. When the precise block structure is not emphasized, we simply write $(u, v)$.
    
    \subsection{Related Work}
    Numerical schemes for problem \eqref{p} have a long history, dating back to the foundational Arrow–Hurwicz method \cite{ref18}. 
    Originally designed for bilinear couplings (i.e., $\Phi(x,y) = \langle Kx, y \rangle$), this method tackles the  saddle-point problem   via alternating proximal gradient steps. 
    Specifically, it computes $x_{n} = \mathrm{Prox}_{\tau_{n}g}(x_{n-1} - \tau_{n}\nabla_{x}\Phi(x_{n-1}, y_{n-1}))$ followed by $y_{n} = \mathrm{Prox}_{\sigma_{n}f^{*}}(y_{n-1} + \sigma_{n}\nabla_{y}\Phi(x_{n}, y_{n-1}))$. 
    While intuitively appealing, the Arrow–Hurwicz scheme is known to not necessarily converge for general bilinear problems \cite{ref16, ref19}. 
    Its convergence is typically guaranteed only under stringent conditions, such as strict strong convexity-concavity \cite{ref8, ref21}, restricted bounded domains \cite{ref3, ref9}, or sufficiently small step sizes \cite{ref20}.
    
    To overcome these fundamental limitations, modern primal-dual algorithms (PDAs) have heavily incorporated extrapolation techniques, notably those introduced  by Chambolle and Pock \cite{ref3, ref22}. 
    While initially developed for bilinear settings, these PDAs have now been extended to more general cases.
    For instance, Zhu et al. \cite{ref23} introduced a class of primal-dual algorithms for problem \eqref{p} that leverage a potential function, Nesterov's acceleration, and adaptive parameter updates, ultimately achieving optimal rates.  
    More recently, Hamedani and Aybat \cite{ref10} presented a new primal-dual algorithm   by introducing a specific momentum term on the dual gradient, where the iterative scheme is given by    \begin{equation}\label{PDA}
        \left\{
        \begin{aligned}
            x_n &= \mathrm{Prox}_{\tau_n g}\bigl(x_{n-1} - \tau_n\nabla_x\Phi(x_{n-1},y_{n-1})\bigr), \\
            z_n &= (1+\delta)\nabla_y\Phi(x_n,y_{n-1}) - \delta\nabla_y\Phi(x_{n-1},y_{n-2}), \\
            y_n &= \mathrm{Prox}_{\sigma_n f^*}\bigl(y_{n-1} + \sigma_n z_n\bigr),
        \end{aligned}
        \right.
    \end{equation}
    where $\delta \in (0,1]$. 
    They also incorporated an adaptive linesearch to dynamically tune $\tau_n$ and $\sigma_n$.
    
    Another highly influential line of research stems from the use of convex combination techniques.  
    Malitsky \cite{ref14} introduced the golden ratio algorithm (GRA), which utilizes a Jacobian-type update augmented by a fully adaptive step size to solve general mixed variational inequalities (MVIs). 
    Since the optimality conditions of \eqref{p} can be cast as an MVI, GRA is theoretically applicable. 
    However, empirical evidence suggests that applying GRA's Jacobian-style updates directly to \eqref{p} yields suboptimal computational efficiency compared to Gauss-Seidel-based PDAs like \eqref{PDA}.   
    To address this performance gap, subsequent works proposed the golden ratio primal-dual algorithm (GRPDA) \cite{ref15, ref24}, integrating the convex combination strategy with alternating updates in bilinear settings. 
    This architecture was recently advanced by the PDAc-L algorithm \cite{ref25}, which efficiently extended the methodology to general convex-concave saddle-point problems via the following iterative scheme:
    \begin{equation}\label{eq:PDAc}
        \left\{
        \begin{aligned}
            z_n &= \frac{\psi-1}{\psi}x_{n-1} + \frac{1}{\psi}z_{n-1}, \\
            x_n &= \mathrm{Prox}_{\tau_n g}\bigl(z_n - \tau_n \nabla_x \Phi(x_{n-1}, y_{n-1})\bigr), \\
            y_n &= \mathrm{Prox}_{\sigma_n f^*}\bigl(y_{n-1} + \sigma_n \nabla_y \Phi(x_n, y_{n-1})\bigr).
        \end{aligned}
        \right.
    \end{equation}
    Notably, PDAc-L expanded the admissible convex combination parameter range $\psi$ to $(1, 1+\sqrt{3})$ and demonstrated linear convergence for strongly convex problems with the type of coupling $\langle H(x), y \rangle$. 
    By incorporating a special convex combination scheme, Chen, Lan and Ouyang \cite{ref27} developed an accelerated primal-dual   method for a class of deterministic and stochastic saddle‑point problems, where the iterative scheme is given by  
    \[\left\{ {\begin{array}{*{20}{l}}
    {\bar x_n^{md} = (1 - \beta _n^{ - 1})x_n^{ag} + \beta _n^{ - 1}{x_n},}\\
    {{y_{n + 1}} = {{{\mathop{\rm argmin}\nolimits} }_{y \in Y}}\langle  - K{{\bar x}_n},y\rangle  + {f^*}(y) + \frac{1}{{{\tau _n}}}y - {y_n}{^2},}\\
    {{x_{n + 1}} = {{{\mathop{\rm argmin}\nolimits} }_{x \in X}}\langle \nabla g(\bar x_n^{md}),x\rangle  + \langle x,{K^ \top }{y_{n + 1}}\rangle  + \frac{1}{{{\eta _n}}}x - {x_n}{^2},}\\
    {x_{n + 1}^{ag} = (1 - \beta _n^{ - 1})x_n^{ag} + \beta _n^{ - 1}{x_{n + 1}},}\\
    {y_{n + 1}^{ag} = (1 - \beta _n^{ - 1})y_n^{ag} + \beta _n^{ - 1}{y_{n + 1}},}\\
    {{{\bar x}_{n + 1}} = {\theta _{n + 1}}({x_{n + 1}} - {x_n}) + {x_{n + 1}}.}
    \end{array}} \right.\]
    Here, $\{\beta_n\}$ is a sequence satisfying $\beta_n>1$ for all $n$.   Notably, their method achieves an optimal convergence rate   while avoiding any smoothing of the objective function, and can handle  scenarios where the feasible region is unbounded  provided that a saddle point exists. 

    
    \subsection{Motivation and Contributions}
    Primal-dual algorithms with convex combinations have exhibited remarkable numerical efficiency, which   has made them  constitute a mainstream class of methods for convex-concave saddle-point problems.  The choice of the convex combination parameter (for example, $\psi$ in \eqref{eq:PDAc}) often has a significant impact on both the theoretical analysis and the numerical performance of the corresponding algorithms.  However, the requirements on this parameter imposed by theory are sometimes inconsistent with those suggested by numerical experiments, which may cause a gap between theoretical analysis and practical performance.  

    Recently, Liu and Liu   \cite{ref11} exploited a strategy of  two convex combinations to  present a  primal-dual algorithm (NPDAL), which has demonstrated outstanding   numerical performance for convex-concave saddle-point with bilinear coupling term. Although the strategy of two convex combinations  can overcome the above limitation on the convex combination parameter, it is currently limited to bilinear coupling structures; whether it works for the general cases remains  unknown.

    To address the limitations on convex combination parameter and   further advance primal-dual algorithms with convex combinations, in this paper, we exploit a new  strategy of two convex combinations  based on that in \cite{ref11} to relax the limitation on the convex combination parameter, and present a new primal-dual algorithm (NPDAL-n) with the two convex combinations and linesearch for general convex-concave saddle-point problem.
    
    The primary contributions of this work are fourfold:
    \begin{enumerate}
        \item To handle general convex-concave saddle-point problems, we extend the NPDAL framework \cite{ref11}. Specifically, we first construct two auxiliary points via convex combinations of previous iterates:
        \begin{equation} \label{znn}
           z_n = \frac{\psi-1}{\psi}x_{n-1} + \frac{1}{\psi}z_{n-1}, \quad \bar{z}_n = (1-b)y_{n-1} + b \bar{z}_{n-1},    
        \end{equation}
        and then form another convex combination of these auxiliary points to obtain the intermediate   points:
        \begin{equation} \label{xmdymd}
        x_{n}^{md} = a x_{n-1} + (1-a) z_n, \quad y_{n}^{md} = b y_{n-1} + (1-b) \bar{z}_{n}.
       \end{equation}
        These intermediate points are subsequently employed in the proximal updates for the primal and dual variables: 
        $$
        \begin{aligned}
            x_n &= \mathrm{Prox}_{\tau g}\bigl(x_n^{md} - \tau \nabla_x\Phi(x_{n-1},y_{n-1})\bigr), \\
            y_n &= \mathrm{Prox}_{\sigma f^*}\bigl(y_n^{md} + \sigma \nabla_y\Phi(x_n,y_{n-1})\bigr).
        \end{aligned}
        $$
        This construction effectively relaxes the restriction on the convex combination parameter.  Under local Lipschitz continuity assumptions on $\nabla_x \Phi$ and $\nabla_y \Phi$, we rigorously prove global pointwise convergence and establish an $\mathcal{O}(1/N)$ ergodic sublinear convergence rate of the proposed  NPDAL-n framework.
    
        \item When the primal function 
        $g$  is strongly convex and the coupling term takes the form  $\Phi(x,y) = h(x) + \langle H(x),y \rangle$, we develop an accelerated variant of NPDAL-n and establish a    $\mathcal{O}(1/N^2)$ convergence rate in terms of the primal-dual gap function.  

        \item In a special scenario where the composite convex optimization problem \eqref{finite_sum_saddle} is transformed into a convex‑concave saddle‑point problem, we demonstrate that NPDAL‑n can simplify  to a linesearch‑free proximal gradient method, thereby entirely avoiding linesearch procedures and global Lipschitz constant knowledge.  
    
        \item   Numerical experiments on   quadratically constrained quadratic programming (QCQP) and sparse logistic regression (SLR) tasks  demonstrate that the proposed algorithms not only significantly reduce runtime and iteration counts, but also consistently outperform several existing state‑of‑the‑art primal‑dual methods.
    \end{enumerate}

    \subsection{Organization}
    The remainder of the paper is organized as follows. Building upon the notation introduced in Section \ref{notation}, Section \ref{section2} formally outlines the fundamental assumptions and provides the necessary analytical preliminaries. In Section \ref{section3}, we formally present a new primal-dual algorithm with two convex combinations and linesearch, followed by detailed theoretical proofs of its global convergence and sublinear convergence rate. Section \ref{section4} shifts focus to the strongly convex regime for nonlinear compositional optimization, presenting the accelerated algorithm alongside its $\mathcal{O}(1/N^2)$ rate guarantee. In Section \ref{section5}, we additionally demonstrate how this framework elegantly simplifies to a fully adaptive proximal gradient method when applied to associated the composite convex optimization formulation \eqref{finite_sum}. Finally, Section \ref{section6} presents comprehensive numerical evaluations on QCQP and SLR problems, and Section \ref{section7} offers concluding remarks.
        
\section{Assumptions and Preliminaries}\label{section2}
    This section presents fundamental mathematical tools and assumptions necessary for our convergence analysis. 
    Let $a, b, c \in \mathbb{R}^p$ and $\lambda \in \mathbb{R}$.    The following   algebraic identities are frequently utilized in the subsequent theoretical analysis:
    \begin{align}
        2\langle a-b, a-c \rangle &= \|a-b\|^2 + \|a-c\|^2 - \|b-c\|^2, \label{identity_inner} \\
        \|\lambda a + (1-\lambda)b\|^2 &= \lambda\|a\|^2 + (1-\lambda)\|b\|^2 - \lambda(1-\lambda)\|a-b\|^2. \label{identity_convex}
    \end{align}
    
    A point pair $(x^*, y^*) \in \mathrm{dom}(g) \times \mathrm{dom}(f^*)$ is designated as a saddle-point of the minimax objective $\mathcal{L}$ if it satisfies 
    $$ 
    \mathcal{L}(x^*, y) \le \mathcal{L}(x^*, y^*) \le \mathcal{L}(x, y^*), \quad \forall (x,y) \in \mathrm{dom}(g) \times \mathrm{dom}(f^*). 
    $$
   The set of all such saddle-points is denoted by $\Omega$. 
    According to the first-order optimality conditions, this set can be explicitly formulated as:
    \begin{equation}
        \Omega = \left\{ (x^*, y^*) \in \mathrm{dom}(g) \times \mathrm{dom}(f^*) \mid -\nabla_x \Phi(x^*, y^*) \in \partial g(x^*), \; \nabla_y \Phi(x^*, y^*) \in \partial f^*(y^*) \right\}.
    \end{equation}
    
    To quantify the convergence metric, we define the primal-dual gap function evaluated at a fixed $(x^*, y^*) \in \Omega$ as follows:
    \begin{equation} \label{J}
        J(x,y) := \mathcal{L}(x, y^*) - \mathcal{L}(x^*, y).
    \end{equation}
    Note that $J(x,y)$ is inherently non-negative and jointly convex with respect to $(x,y)$.
    
    Throughout this paper, our theoretical framework is grounded on the following blanket assumptions.
    
    \begin{assumption}\label{assump:2.1}
        The solution set is non-empty, i.e., $\Omega \neq \emptyset$. 
        Additionally, the domain condition $\mathrm{dom}(g) \times \mathrm{dom}(f^*) \subseteq \mathrm{dom}(\Phi)$ holds, and the optimal value $\mathcal{L}(x^*, y^*)$ is finite.
    \end{assumption}
    
    \begin{assumption}\label{assump:2.2}
        The continuous coupling mapping $\Phi: \mathrm{dom}(g) \times \mathrm{dom}(f^*) \to \mathbb{R}$ satisfies two structural conditions:
        \begin{enumerate}
            \item[(i)] \textbf{(convexity and concavity)} The function $\Phi$ is convex and differentiable with respect to its first argument over $\mathrm{dom}(g)$ for any fixed $y \in \mathrm{dom}(f^*)$. Conversely, it is concave and differentiable with respect to its second argument over $\mathrm{dom}(f^*)$ for any fixed $x \in \mathrm{dom}(g)$.
            
            \item[(ii)] \textbf{(Local Lipschitz Gradients)} Given any bounded regions $\mathcal{X} \subset \mathbb{R}^p$ and $\mathcal{Y} \subset \mathbb{R}^q$, there exist $L_{xx} \ge 0$, $L_{yy} \ge 0$, and $L_{xy} > 0$ such that:
            \begin{align*}
                \|\nabla_y \Phi(x, y) - \nabla_y \Phi(x, \tilde{y})\| &\le L_{yy} \|y - \tilde{y}\|, \\
                \|\nabla_x \Phi(x, y) - \nabla_x \Phi(\tilde{x}, \tilde{y})\| &\le L_{xx} \|x - \tilde{x}\| + L_{xy} \|y - \tilde{y}\|,
            \end{align*}
            hold for all admissible points $x, \tilde{x} \in \mathcal{X} \cap \mathrm{dom}(g)$ and $y, \tilde{y} \in \mathcal{Y} \cap \mathrm{dom}(f^*)$.
        \end{enumerate}
    \end{assumption}
    
    \begin{remark}\label{rem:2.1}
        The local Lipschitz constants ($L_{xx}, L_{yy}, L_{xy}$) introduced above are utilized solely for establishing theoretical convergence. 
        Our algorithmic design remains entirely parameter-free with respect to these constants. 
        It suffices that these bounds exist on the bounded trajectories $\mathcal{X} \times \mathcal{Y}$ dynamically generated by the iterations.
    \end{remark}
    
    \begin{assumption}\label{assump:2.3}
        The proximal operators associated with the constituent functions $g$ and $f^*$ are computationally tractable.
    \end{assumption}
    
    We conclude this section by cataloging three fundamental auxiliary lemmas utilized in our subsequent proofs. 
    \begin{lemma}[\cite{ref25}]\label{fact2.2}
        Let $h: \mathbb{R}^m \to (-\infty, +\infty]$ be an extended real-valued, closed, proper, and $\gamma$-strongly convex function with modulus $\gamma \ge 0$. 
        For any given $x \in \mathbb{R}^m$ and scalar $\tau > 0$, $z = \mathrm{Prox}_{\tau h}(x)$ if and only if 
        $$ 
        h(y) \ge h(z) + \frac{1}{\tau} \langle x - z, y - z \rangle + \frac{\gamma}{2} \|y - z\|^2, \quad \forall y \in \mathbb{R}^m. 
        $$
    \end{lemma}
    
    \begin{lemma}[\cite{ref25}]\label{fact2.3}
        Consider two non-negative real sequences $\{u_n\}$ and $\{v_n\}$. 
        If there exists a decay factor $\varepsilon \in (0,1)$ such that $u_{n+1} \le \varepsilon u_n + v_n$ for all $n \ge 1$, and $\sum_{n=1}^\infty v_n < \infty$, then it follows that $\sum_{n=1}^\infty u_n < \infty$.
    \end{lemma}
    
    \begin{lemma}[\cite{ref25}]\label{fact2.4}
        For any real numbers $P, Q \in \mathbb{R}$ and non-negative weights $u, v \in \mathbb{R}$ satisfying $u+v > 0$, there holds:
        $$ 
        \frac{uv}{u+v}(P+Q)^2 \le uP^2 + vQ^2. 
        $$
    \end{lemma}
 
\section{The Proposed Algorithm and Convergence Analysis}\label{section3}

Based on the two convex combinations \eqref{znn}
 and \eqref{xmdymd}, we propose a new primal-dual algorithm with linesearch for    the convex-concave saddle-point problem \eqref{p}, and establish its global pointwise convergence and   sublinear ergodic rate.
    
    To streamline the presentation of the algorithm, we first define several auxiliary quantities. At each iteration $n$, we track the gradient variation and the linearization error associated with the function $\Phi$, denoted respectively by:
    \begin{align}
        \theta_n &:= \nabla_x \Phi(x_n, y_n) - \nabla_x \Phi(x_{n-1}, y_{n-1}), \label{tt1} \\
        \Phi_n^y &:= \Phi(x_n, y_{n-1}) + \langle \nabla_y \Phi(x_n, y_{n-1}),\, y_n - y_{n-1}\rangle - \Phi(x_n, y_n). \label{tt2}
    \end{align}
    
    Furthermore, we introduce the algorithmic parameters $\omega$ and $\omega_1$, which depend on the hyperparameters $\psi \in (1, 1+\sqrt{3})$, $\xi$, $\varphi$, and $a \in [0,1)$:
    \begin{equation}\label{eq:omega_defs}
        \omega(\xi,\varphi) := 2\psi - \xi - \frac{\psi^{3}\varphi}{1+\psi}, \quad \text{and} \quad \omega_1 := (1-a)\omega(\xi,\varphi) - \frac{a}{\xi_1}.
    \end{equation}
    
    To guarantee the stability of the method, these parameters must be chosen from the following feasible parameter space:
    \begin{equation}\label{eq:theta_space}
        \Theta_{\psi} = \left\{ (\xi,\xi_1,\varphi,\omega_1) \mid \xi > 0, \; 1/\varphi>\xi_1>0, \; \psi>\varphi > 1, \; \omega_1 > 0 \right\}.
    \end{equation}
    
    It is straightforward to verify that the admissible set $\Theta_{\psi}$ is non-empty. For any $\psi \in (1,1+\sqrt{3})$, $\xi$ and $\varphi$ can be flexibly selected in the region $\Theta_{\psi}$. Utilizing these predefined quantities, the complete procedural steps of our approach are detailed below.

    \begin{algorithm}[H] 
    \caption{ NPDAL-n: New PDA with two convex combinations and linesearch for problem \eqref{p} }
    \begin{algorithmic} 
        \State \textbf{Initialization:} Choose $\psi\in(1,1+\sqrt{3}),\ (\xi,\xi_1,\varphi,\omega_1) \in \Theta_\psi$, $\tau_{max}>0$, $\nu \in (0,1),\ \mu \in (0,1),\ \eta \in [0,1),\ a\in (0,1),\ b\in [0,1) $, and an integer $M\geq 1$. Choose $x_0 \in \textbf{dom}(g),y_0 \in \mathbf{dom}(f^*) $, $\beta>0$ and $\tau_0 \in (0,\tau_{max}]$. Set $z_0=x_0$, $\bar{z}_0=y_0$ and $n=1$.
        \State \textbf{Main Iteration:}
            \State Step 1. Compute 
            \begin{align}
                z_{n} &= \frac{\psi-1}{\psi} x_{n-1} +\frac{1}{\psi} z_{n-1}, \quad
                x_{n}^{md} = a x_{n-1} +(1-a)z_{n},   \label{zx1}\\
                x_{n} &= \text{Prox}_{\tau_{n-1} g}\big(x_{n}^{md}-\tau_{n-1} \nabla_x \Phi(x_{n-1},y_{n-1})\big).\label{x1}
            \end{align}    
            \State Step 2. Compute
            \begin{align}
                \bar{z}_n =  (1-b) y_{n-1}+b \bar{z}_{n-1}, \quad y_{n}^{md} =b y_{n-1} + (1-b) \bar{z}_n. \label{zy1}
            \end{align} 
        \State Step 3. \textbf{Linesearch.} Set $\tau =\min\{\varphi\tau_{n-1},\tau_{max} \}$ and perform 
            \begin{align}        
                y_{n}=\text{Prox}_{\beta\tau_n f^*}(y_{n}^{md}+\beta\tau_n \nabla_y \Phi(x_{n},y_{n-1})),\label{y1}
            \end{align}
        \State \hspace{3.2em}  $\tau_{n} :=\tau\mu^i  \;\text{for} \; i=0,1,\cdots $     until the following condition    is satisfied:
            \begin{align}
                &\ \frac{\tau_n \tau_{n-1}}{(1-a)\xi}\|\theta_n\|^2 + 2\tau_n \Phi_n^y \leq \nu r_n +(1-\nu)c_n. \label{lin} \\
                \text{Here,} \quad r_n &=\omega_1 \delta_{n-1} \|x_n - x_{n-1}\|^2 +\frac{b}{\beta}\|y_n - y_{n-1}\|^2 + \frac{1}{\beta}(1-b)(1+b)\|y_n - \bar{z}_n\|^2,\nonumber \\
                c_n&=(\eta/|I_n|)\sum_{i \in |I_n|}r_i  \ \text{and} \quad I_n=\{n-1, n-2,..., \max\{n-M,1\}\}.\label{rn}
            \end{align} 
        \State Step 4. Set $\delta_n=\tau_n/\tau_{n-1}$, $n \leftarrow n+1$ and go to Step 1.
    \end{algorithmic}  \label{Alg1}
    \end{algorithm} 

    \begin{remark}
         In the special case where  $a=0  \;\text{and} \;b=0$,  NPDAL‑n reduces to PDAc‑L \cite{ref25}.   
    \end{remark}

    \begin{remark}
    In general, a large weight on the latest iterate (for example, the coefficient of $y_{n-1}$ in $y_{n}^{md}$, or that of $x_{n-1}$ in $x_{n}^{md}$) is preferred in numerical experiments, while a small weight is sometimes required for theoretical analysis. This discrepancy may lead to a gap between theory and practice.

    For $y_{n-1}$ in $y_{n}^{md}$,  we obtain that  
    $$
    y_{n}^{md}=  b y_{n-1}+(1-b) \bar z_n = (1-b+b^2)y_{n-1}+b(1-b) \bar z_{n-1}.
    $$
    Denote $\phi_1 (b) =1-b+b^2 $. We know that the weight   assigned to the latest iterative point $y_{n-1}$ in $ y_{k}^{md}$ is $ \phi_1 (b ) =1-b+b^2 $, giving $\phi_1 (1)=1$ and $\phi_1 (0)=0$, and $\mathop {\min }\limits_{0 \le b \le 1} \phi \left( b \right) = 0.75$, which imply that the weight for $y_{n-1}$ satisfies $\phi_b \left(b \right) \ge 0.75$ for any $b \in (0,1)$.

    For $x_{n-1}$  in $x_{n}^{md}$,   we obtain that  
    \[
    x_n^{md} = a{x_{n - 1}} + (1 - a){z_n} = \left[ {a + (1 - a)\frac{{\psi  - 1}}{\psi }} \right]{x_{n - 1}} + \frac{{1 - a}}{\psi }{z_{n - 1}}.
    \]
    Denote ${\phi _2}\left( {a,\psi } \right) = a + (1 - a)\frac{{\psi  - 1}}{\psi }$. It follows  that $\frac{{\psi  - 1}}{\psi }<{\phi _2}\left( {a,\psi } \right)<1$.   In PDAc-L \cite{ref25}, the weight for the latest iterate $x_{n-1}$  in $z_{n} $ in  \eqref{eq:PDAc} is $\frac{{\psi  - 1}}{\psi } \in \left( {0,\frac{{\sqrt { 3 } }}{{\sqrt { 3 }  + 1}}} \right)$, which implies weight for the latest iterate $x_{n-1}$ is less  than 0.64. Note that for any $a \in (0,1)$, the weight assigned to the latest iterate  $x_{n-1}$  in NPDAL-n is greater than that in PDAc-L \cite{ref25}. In addition, this weight is often substantially larger, since $\psi $ is commonly set to a large value.  

    Based on the above analysis, the weights assigned to the latest iterates  (e.g., $y_{n-1}$ in $y_n^{md}$ or  $x_{n-1}$  in  $x_n^{md}$) of NPDAL-n are relatively  large for any   $a,b \in (0,1)$. Consequently, the permissible range of the convex combination parameters is mainly determined   by theoretical considerations, with little regard for numerical performance. 
    \end{remark}
    
    \subsection{Fundamental Properties of NPDAL-n}

    In this subsection, we establish several essential characteristics of the   sequence $\{(z_n, \bar{z}_n, x_n, y_n)\}$ and analyze the behavior of the adaptive step sizes $\tau_n$. These foundational results are critical for our subsequent convergence analysis.
    
    Recalling the local Lipschitz gradient condition specified in Assumption \ref{assump:2.2}, we can systematically bound the linearization error of the concave component. Specifically, for any fixed $x \in \mathrm{dom}(g)$ and any dual variables $y, \tilde{y}$ constrained within a bounded subset $Y \subset \mathrm{dom}(f^*)$, the following double inequality holds:
    \begin{equation}\label{t}
        -\frac{L_{yy}}{2}\|y-\tilde{y}\|^2 \le \Phi(x, y) - \Phi(x, \tilde{y}) - \langle \nabla_y \Phi(x, \tilde{y}), y-\tilde{y} \rangle \le 0.
    \end{equation}
    Building upon this structural bound, we now introduce three key auxiliary lemmas.
        
    \begin{lemma} \label{le1}
    Suppose Assumption \ref{assump:2.1}, \ref{assump:2.2} and \ref{assump:2.3} hold, and let $\left\{ {{x_n},{y_n}} \right\}$ be the sequence generated by NPDAL-n. Then,
       we have that
        \begin{equation} \label{t0}
			\begin{split}
				\tau_n J(x_n, y_n) \leq &\langle x_{n+1} - x_{n+1}^{md}, x^* - x_{n+1} \rangle + \delta_n \langle x_n - x_{n}^{md}, x_{n+1} - x_n \rangle + \tau_n \langle \theta_n,x_n-x_{n+1} \rangle  \\
                &+\tau_n \Phi_n^y+ \frac{1}{\beta} \langle y_n - y_{n}^{md}, y^* - y_{n} \rangle,
			\end{split}
		\end{equation}
        where $\theta_n$, $\Phi_n^y$ and $J(\cdot,\cdot)$ are defined respectively in (\ref{tt1}), (\ref{tt2}) and (\ref{J}).  
    \end{lemma}

    \begin{proof}
        It follows from (\ref{x1}), (\ref{y1}) and Lemma \ref{fact2.2} that
        \begin{align}
            \tau_n (g(x_{n+1}) - g(x^*))  &\leq \langle x_{n+1} - x_{n+1}^{md}+ \tau_n \nabla_x \Phi(x_n, y_n), x^* - x_{n+1}\rangle , \label{t1}\\
             \tau_{n-1} (g(x_{n}) - g(x_{n+1}))  &\leq \langle x_{n} - x_n^{md}+ \tau_{n-1} \nabla_x \Phi(x_{n-1}, y_{n-1}), x_{n+1}- x_{n} \rangle ,  \label{t2} \\
             \tau_{n} (f^*(y_{n}) - f^*(y^*))  & \leq  
             \langle \frac{1}{\beta}(y_{n} - y_{n}^{md})- \tau_{n} \nabla_y \Phi(x_{n}, y_{n-1}),y^*- y_{n} \rangle.\label{t3}
        \end{align}
       Multiplying  \eqref{t2} by $\delta_n=\tau_{n}/\tau_{n-1}$ yields
        \begin{align}
            \tau_{n} (g(x_{n}) - g(x_{n+1})) \leq \langle \delta_n( x_{n} - x_n^{md})+ \tau_{n} \nabla_x \Phi(x_{n-1}, y_{n-1}), x_{n+1}- x_{n} \rangle . \label{t4} 
        \end{align}
        It follows from the right-hand side of (\ref{t}) that 
        \begin{align*}
            -\langle \nabla_y \Phi(x_n, y_{n-1}),\, y^* - y_n \rangle
            &= \langle \nabla_y \Phi(x_n, y_{n-1}),\, y_n - y_{n-1} \rangle 
            + \langle \nabla_y \Phi(x_n, y_{n-1}),\, y_{n-1} - y^* \rangle  \\
            &\leq \langle \nabla_y \Phi(x_n, y_{n-1}),\, y_n - y_{n-1} \rangle
            + \Phi(x_n, y_{n-1}) - \Phi(x_n, y^*).
        \end{align*}
        Taking the sum of (\ref{t1}), (\ref{t3}) and (\ref{t4})  and using the above inequality and the definition of $J(\cdot,\cdot)$ in (\ref{J}), we obtain  
        \begin{equation} \label{t5}
			\begin{split}
    			\tau_n J(x_n,y_n) 
                \le & \langle x_{n+1}-x_{n+1}^{md},\, x^*-x_{n+1}\rangle +\delta_n \langle x_{n} - x_{n}^{md},  x_{n+1}- x_{n} \rangle \\
                &+\frac{1}{\beta}\langle y_n-y_{n}^{md},\,y^*-y_n\rangle
                +\tau_n G_n,
			\end{split}
		\end{equation}
        where
        \begin{align*}
            G_n &= \langle \nabla_x \Phi(x_n, y_{n}),\,x^*-x_{n+1}\rangle
            +\langle \nabla_x \Phi(x_{n-1},y_{n-1}),\,x_{n+1}-x_n\rangle\\
            &\quad +\langle \nabla_y \Phi(x_n,y_{n-1}),\,y_n-y_{n-1} \rangle + \Phi(x_n,y_{n-1}) - \Phi(x^*,y_n).
        \end{align*}
        Combining (\ref{tt1})-(\ref{tt2}) and the convexity of $\Phi$ in $x$, we have that 
        \begin{align*}
            G_n &= \langle \theta_n,x_n-x_{n+1}\rangle +\Phi_n^y
            +(\Phi(x_n,y_n)+\langle\nabla\Phi(x_n,y_n),x^*-x_n\rangle
            -\Phi(x^*,y_n)\big)\\
            &\le \langle\theta_n,x_n-x_{n+1}\rangle +\Phi_n^y.
        \end{align*}
       Together with (\ref{t5}), it follows that (\ref{t0}).

    \end{proof}

    For any $(x^*,y^*)\in\Omega$, we define the   two sequences  $\left\{ {{a_n}} \right\}\;{\rm{and }}\left\{ {{b_n}} \right\}$ by 
    \begin{equation} \label{an1}
        \left\{
        \begin{split}
            a_n&=a\|x_{n}- x^*\|^2 +(1-a)\frac{\psi}{\psi-1}\| z_{n+1}-x^* \| ^2+(a+\omega_1 \delta_{n-1})\| x_{n} - x_{n-1}\|^2  \\
            &\quad+ \frac{b}{\beta} \|y_{n-1} - y^*\|^2 + \frac{1}{\beta} \|\bar{z}_n - y^*\|^2,\\
            b_n&=\omega_1 \delta_{n-1}\|x_{n} - x_{n-1}\|^2 + \frac{b}{\beta}\| y_n - y_{n-1}\|^2 + \frac{1}{\beta}(1-b)(1+b)\|y_n - \bar{z}_n\|^2 \\
            &\quad-\frac{\tau_n \tau_{n-1}}{(1-a)\xi }\| \theta_n \|^2 - 2\tau_n \Phi_n^y.
        \end{split}
        \right .
    \end{equation}
    
    \begin{lemma}\label{le2}   Suppose Assumption \ref{assump:2.1}, \ref{assump:2.2} and \ref{assump:2.3} hold, and let $\left\{ {{x_n},{y_n}} \right\}$ be the sequence generated by NPDAL-n. Then, we have
         $2\tau_n J(x_n, y_n) +a_{n+1} \leq a_n-b_n$, where $a_n$ and $b_n$ are defined in (\ref{an1}).
    \end{lemma}
    
    \begin{proof}
        It follows from the definitions of $x_n^{md}$ and $z_n$ that $x_n-z_n=\psi(x_n-z_{n+1})$. By Lemma \ref{le1}, we obtain
        \begin{equation} \label{eq11}
			\begin{split}
				\tau_n J(x_n,y_n) 
                \le & a\langle x_{n+1}- x_{n},\, x^*-x_{n+1}\rangle +(1-a)\langle x_{n+1}-z_{n+1} ,\, x^*-x_{n+1}\rangle  \\
                &+a\delta_n \langle x_{n} - x_{n-1}, x_{n+1}- x_{n} \rangle +(1-a)\psi\delta_n \langle x_{n} - z_{n+1}, x_{n+1} - x_{n} \rangle \\
                &+\frac{1}{\beta}\langle y_n-y_{n}^{md},\,y^*-y_n\rangle +\tau_n \langle\theta_n,x_n-x_{n+1}\rangle +\tau_n\Phi_n^y.
			\end{split}
		\end{equation}
       It follows from  Young's inequality  that
       \begin{equation} \label{eqYoung1}
        2a\delta_n \langle x_{n} - x_{n-1}, x_{n+1}- x_{n} \rangle \le \xi_1 a\delta_n\| x_{n} - x_{n-1}\|^2+\frac{a\delta_n}{\xi_1}\| x_{n+1}- x_{n}\|^2 ,
        \end{equation}
        where $\xi_1 \in (0,\frac{1}{\varphi})$. Combining with the definition of $\delta_n$ yields that   $\xi_1 a \delta_n \le \xi_1 a \varphi <a$.
        
        Substituting  \eqref{eqYoung1} into \eqref{eq11} and using  \eqref{identity_inner}, we can  obtain
        \begin{equation} \label{eq12}
			\begin{split}
				& \quad 2\tau_n J(x_n,y_n) + a\| x_{n+1} - x_{n}\|^2+a\|x_{n+1}- x^*\|^2  \\
                &+(1-a)\psi\delta_n \|x_{n} - z_{n+1}\|^2+ (1-a)\psi\delta_n\| x_{n+1}- x_{n}\|^2  \\
                & + (1-a)(1-\psi\delta_n)\|x_{n+1}-z_{n+1}\|^2+(1-a)\|x_{n+1}-  x^*\|^2 \\
                &+\frac{1}{\beta}\|y_n-y_{n}^{md}\|^2+\frac{1}{\beta}\|y_n-y^*\|^2  \\
                \le & a\|x_{n}- x^*\|^2 +(1-a)\|z_{n+1}-  x^*\|^2  +2\tau_n \langle\theta_n,x_n-x_{n+1}\rangle +2\tau_n\Phi_n^y  \\
                &+ a\| x_{n} - x_{n-1}\|^2+\frac{a\delta_n}{\xi_1}\| x_{n+1}- x_{n}\|^2 +\frac{1}{\beta}\|y_n^{md}-y^*\|^2.
			\end{split}
		\end{equation}
        By   \eqref{identity_convex}, we obtain
        \begin{align}
            \| y_n - y_{n}^{md}\|^2 & = b\| y_n - y_{n-1}\|^2+(1-b)\| y_n - \bar
            z_{n}\|^2-b(1-b)\| y_{n-1} - \bar
            z_{n}\|^2,\nonumber \\
            \| y_{n}^{md}- y^*\|^2 & = b\|y_{n-1} - y^*\|^2 + (1-b)\|\bar{z}_n - y^*\|^2-b(1-b)\|y_{n-1} - \bar{z}_n\|^2 , \nonumber \\
            \| y_n - y^*\|^2 & = \frac{1}{1-b}\| \bar{z}_{n+1} - y^*\|^2-\frac{b}{1-b}\| \bar{z}_{n} - y^*\|^2 + \frac{b}{(1-b)^2}\| \bar{z}_{n+1} -\bar{z}_{n}\|^2  \nonumber \\
            &=\frac{1}{1-b}\| \bar{z}_{n+1} - y^*\|^2-\frac{b}{1-b}\| \bar{z}_{n} - y^*\|^2 + b\| y_{n} -\bar{z}_{n}\|^2,\nonumber\\
            \| x_{n+1}-x^* \| ^2 &=\frac{\psi}{\psi-1}\| z_{n+2}-x^* \| ^2 -\frac{1}{\psi-1}\| z_{n+1}-x^* \| ^2+\frac{\psi}{(\psi-1)^2}\| z_{n+2}-z_{n+1} \|^2  \nonumber \\
            &=\frac{\psi}{\psi-1}\| z_{n+2}-x^* \| ^2 -\frac{1}{\psi-1}\| z_{n+1}-x^* \| ^2+\frac{1}{\psi}\| x_{n+1}-z_{n+1} \|^2.\nonumber
        \end{align}
        The last two identities employ relations $\bar{z}_{n+1} -\bar{z}_{n} = (1-b)(y_{n} -\bar{z}_{n})$ and $z_{n+2}-z_{n+1} = \frac{\psi - 1}{\psi}(x_{n+1}-z_{n+1})$, respectively. Therefore, we deduce that
        \begin{equation} \label{eq13}
			\begin{split}
                & \quad 2\tau_n J(x_n,y_n) +a\| x_{n+1} - x_{n}\|^2+a\|x_{n+1}- x^*\|^2 +(1-a)\frac{\psi}{\psi-1}\| z_{n+2}-x^* \| ^2 \\
                & +(1-a)\bigg(\psi\delta_n \| z_{n+1}-x_{n} \|^2+ \psi\delta_n\| x_{n+1}- x_{n}\|^2 + \left(1+\frac{1}{\psi}-\psi\delta_n\right)\|x_{n+1}-z_{n+1}\|^2 \bigg) \\
                & +\frac{1}{\beta}(1-b)(1+b)\| y_n - \bar z_{n}\|^2 + \frac{b}{\beta}\| y_n - y^*\|^2+\frac{1}{\beta}\| \bar{z}_{n+1} - y^*\|^2 \\
                \le & a\|x_{n}- x^*\|^2 +(1-a)\frac{\psi}{\psi-1}\| z_{n+1}-x^* \| ^2 +\frac{a\delta_n}{\xi_1}\| x_{n+1}- x_{n}\|^2 -\frac{b}{\beta}\| y_n - y_{n-1}\|^2 \\
                & + \frac{b}{\beta}\|y_{n-1} - y^*\|^2 + \frac{1}{\beta}\|\bar{z}_n - y^*\|^2 + 2\tau_n \langle\theta_n,x_n-x_{n+1}\rangle +2\tau_n\Phi_n^y+ a\| x_{n} - x_{n-1}\|^2.
            \end{split}
		\end{equation}
        Let  $u=\psi \delta_n >0$, $v=1+\frac{1}{\psi}-\psi\delta_n$, $P=\| z_{n+1}-x_{n} \| $ and  $Q=\|x_{n+1}-z_{n+1}\| $. Then we have    $u+v=1+\frac{1}{\psi}>0$, which together with  Lemma \ref{fact2.4} and  $\| x_{n+1}- x_{n}\|  \le P + Q$ implies that
        $$\psi \delta_n(1-\frac{\psi^2 \delta_n}{1+\psi}) \| x_{n+1}- x_{n}\|^2 \le \psi \delta_n\| z_{n+1}-x_{n} \|^2+(1+\frac{1}{\psi}-\psi\delta_n)\|x_{n+1}-z_{n+1}\|^2.$$
        By Young's inequality   and   $\xi>0$, we have 
        $$
        2\langle \theta_n,x_n-x_{n+1} \rangle \leq \frac{(1-a)\xi}{\tau_{n-1}}\|x_n-x_{n+1}\|^2+\frac{\tau_{n-1}}{(1-a)\xi}\| \theta_n \|^2.
        $$ 
        Substituting this above inequality back into \eqref{eq13}, we obtain
        \begin{equation} \label{eq14}
			\begin{split}
                &\quad 2\tau_n J(x_n,y_n) \\
                &+ a\| x_{n+1} - x_{n}\|^2 + a\| x_{n+1}- x^*\|^2 + (1-a)\frac{\psi}{\psi-1}\| z_{n+2}-x^* \|^2 \\
                & + \delta_n \bigg((1-a) \left(2\psi -\xi-\frac{\psi^3 \delta_n}{1+\psi}\right)-\frac{a}{\xi_1} \bigg) \| x_{n+1}- x_{n}\|^2 \\
                &  + \frac{b}{\beta}\| y_n - y^*\|^2 + \frac{1}{\beta}\| \bar{z}_{n+1} - y^*\|^2 + \frac{1}{\beta}(1-b)(1+b)\| y_n - \bar z_{n}\|^2\\
                \le &  a\| x_{n} - x_{n-1}\|^2 + a\|x_{n}- x^*\|^2 + (1-a)\frac{\psi}{\psi-1}\| z_{n+1}-x^* \|^2  \\
                & + \frac{b}{\beta}\|y_{n-1} - y^*\|^2 +  \frac{1}{\beta} \|\bar{z}_n - y^*\|^2 + \frac{\tau_n\tau_{n-1}}{(1-a)\xi}\| \theta_n \|^2 + 2\tau_n\Phi_n^y - \frac{b}{\beta}\| y_n - y_{n-1}\|^2.
            \end{split}
		\end{equation}
        By $2\psi -\xi-\dfrac{\psi^3 \delta_n}{1+\psi} \ge \omega$, $\omega_1=(1-a)\omega-\dfrac{a}{\xi_1} >0$, $\delta_{n} \leq \varphi$ and the definitions of $a_n$ and $b_n$ in \eqref{an1}, we can obtain
        $$
        2\tau_n J(x_n, y_n) +a_{n+1} \leq a_n-b_n.
        $$
    \end{proof}
    
    To establish the lower bound of the step size sequence, we analyze the threshold under which the linesearch condition \eqref{lin} is guaranteed to hold. By utilizing the Lipschitz continuity of the gradients and substituting the corresponding bounds into the linesearch discrepancy, we identify the critical values that make the coefficients of both $\|x_n - x_{n-1}\|^2$ and $\|y_n - y_{n-1}\|^2$ positive. By taking the minimum of these localized thresholds to ensure simultaneous compliance, we define the uniform step size lower bound $\underline{\tau} > 0$ as follows:
    \begin{equation}\label{tau-}
        \underline{\tau} := \min \left\{ \frac{(1-a)\nu\xi\omega_1}{2L_{xx}^2\tau_{\max}}, \frac{b\nu(1-a)\xi}{\beta(2L_{xy}^2\tau_{\max} + L_{yy}(1-a)\xi)} \right\}.
    \end{equation}

    \begin{lemma}\label{le3}
    Suppose that Assumptions 1--3 hold, and let $\{(x_n, y_n, z_n, \bar{z}_n)\}$ be the sequence generated by Algorithm 1. Then the following claims hold:
        \begin{enumerate}
            \item[(i)] The linesearch step in Algorithm \ref{Alg1} always terminates, i.e., the step size sequence $\{\tau_n\}$ is well-defined;
            \item[(ii)] The sequences $\{x_n\}$, $\{y_n\}$, $\{z_n\}$, and $\{\bar{z}_n\}$ are all bounded ; hence,  the intermediate sequences $\{x_n^{md}\}$ and $\{y_n^{md}\}$ are also bounded;
            \item[(iii)] There exists a positive constant $\underline{\tau} > 0$ such that if the trial step size satisfies $\tau \le \underline{\tau}$, then the linesearch condition \eqref{lin} is satisfied;
            \item[(iv)] If $\tau_0 \ge \underline{\tau}$ and $\tau_{max} \ge \mu \underline{\tau}$, then the step size sequence $\{\tau_n\}$ and ${\delta_n}$ are strictly bounded away from zero. Indeed, it holds that $\tau_n \ge \mu \underline{\tau} > 0$ and $\delta_n \ge \mu \underline{\tau}/\tau_{max}$ for all $n \ge 1$.
        \end{enumerate}
    \end{lemma}

    \begin{proof}
        (i) For any $n\ge 1$, we define
        \begin{equation} \label{eq:def_ylambda}
            \left.
            \begin{aligned}
                y_n(\lambda) &:= \operatorname{Prox}_{\beta \lambda f^*}\!\bigl(y_{n-1}+\beta \lambda \nabla_y\Phi(x_n,y_{n-1})\bigr), \quad \lambda>0, \\
                \theta_n(\lambda) &:= \nabla_x \Phi\bigl(x_n,y_n(\lambda)\bigr)-\nabla_x \Phi(x_{n-1},y_{n-1}), \\
                \Phi_n^y(\lambda) &:= \Phi(x_n,y_{n-1})+\bigl\langle \nabla_y \Phi(x_n,y_{n-1}),\,y_n(\lambda)-y_{n-1}\bigr\rangle-\Phi\bigl(x_n,y_n(\lambda)\bigr).
            \end{aligned}
            \right\}
        \end{equation}
        By $\tau=\min\{\varphi\tau_{n-1},\tau_{\max}\}$ and the non-decreasing property of the proximal gradient step length with respect to the step size parameter $\lambda$, we know that
        $$
        \|y_n(\lambda)-y_{n-1}\|\le r:=\|\,y_n(\tau)-y_{n-1}\|< +\infty 
        $$ 
        holds  for any $\lambda\in(0,\tau].$ 
        This implies that the curve $\{y_n(\lambda):\lambda\in(0,\tau]\}$ lies in the closed ball $B[y_{n-1};r]$. Assume, by contradiction, that the linesearch procedure defined in Algorithm \ref{Alg1} fails to terminate at the $n$-th iteration. Then, for all $i=0,1,2,\dots$ and $\lambda=\tau\mu^i$, we have
        \begin{align}
            \frac{\lambda \tau_{n-1}}{(1-a)\xi}||\theta_n||^2+2\lambda \Phi_n^y > \nu r_n +(1-\nu)c_n\geq \nu r_n, \label{le31}
        \end{align}
        where $r_n$ and $c_n$ are given in (\ref{rn}). Since $y_n(\lambda)\in B[y_{n-1};r]$ for all $\lambda=\tau\mu^i$ with $i=0,1,2,\dots$, it follows from Assumption \ref{assump:2.2} (ii) and the inequality on the left-hand side of (\ref{t}) that 
        \begin{equation}
        \|\theta_n(\lambda)\|^2 \leq 2L_{xx}^2\|x_n-x_{n-1}\|^2 +2L_{xy}^2\|y_n(\lambda)-y_{n-1}\|^2 \quad\text{and}\quad \Phi_n^y\leq \frac{L_{yy}}{2}\|y_n(\lambda)-y_{n-1}\|^2.     \label{xyL}
        \end{equation}
       Together with (\ref{le31}) and $\lambda=\tau\mu^i$, we obtain
        \begin{align*}
            &\frac{2\tau \mu^i}{(1-a)\xi}\tau_{n-1}L_{xx}^2\|x_n-x_{n-1}\|^2+\frac{\tau \mu^i}{(1-a)\xi}(2\tau_{n-1}L_{xy}^2+(1-a)\xi L_{yy})\|y_n(\lambda)-y_{n-1}\|^2  \\
            & > \nu r_n = \nu \bigg(  \omega_1 \delta_{n-1} \| x_n - x_{n-1}\|^2 +\frac{b}{\beta}\| y_n - y_{n-1}\|^2 + \frac{1}{\beta}(1-b)(1+b)\|y_n - \bar{z}_n\|^2 \bigg)\\
            & \ge \nu \bigg(  \omega_1 \delta_{n-1} \| x_n - x_{n-1}\|^2 +\frac{b}{\beta}\| y_n - y_{n-1}\|^2 \bigg),
        \end{align*}  
        which implies
        $\frac{\tau \mu^i}{(1-a)\xi}2\tau_{n-1}L_{xx}^2 >  \nu \omega_1 \delta_{n-1}$
        or $\tau \mu^i(2\tau_{n-1}L_{xy}^2+(1-a)\xi L_{yy})/(1-a) \xi > \nu b/\beta$.
        This is impossible since $\mu^i\to 0$ as $i\to\infty$, which indicates that the linesearch procedure must terminate.
        
        (ii) By \eqref{lin} and the definition of $b_n$ given in \eqref{an1}, we obtain  
        $$b_n \ge (1-v)(r_n-c_n)=(1-v) \bigg( r_n-\frac{\eta}{|I_n|}\sum_{i\in I_n}r_i\bigg),$$ 
        where $I_n=\{n-1, n-2,...,max\{n-M,1\}\}$ and $M\ge 1 $ is an integer.
        Hence, for all $k\ge M+1$ we have
        \begin{equation} \label{le32}
			\begin{split}
                \sum_{n=M+1}^{k}b_n &\ge (1-v) \bigg(\sum_{n=M+1}^{k}r_n-\frac{\eta}{M}\sum_{n=M+1}^{k}\sum_{i=n-M}^{n-1}r_i \bigg)  \\
                &= (1-v)(1-\eta)\sum_{n=M+1}^{k}r_n+(1-v)\eta \bigg(\sum_{n=M+1}^{k}r_n-\frac{1}{M}\sum_{n=M+1}^{k}\sum_{i=n-M}^{n-1}r_i \bigg) \\
                &\ge (1-v)(1-\eta)\sum_{n=M+1}^{k}r_n-(1-v)\eta\sum_{n=1}^{M}r_n.
            \end{split}
		\end{equation}
        Since $J(x_n,y_n)\ge 0$, it follows  that $a_{n+1}\le a_n-b_n$ for all $n\ge 1$.
        It follows from  $r_n\ge 0$,  $\nu \in(0,1)$,  $\eta\in[0,1)$ and (\ref{le32}) that
         $$a_{k+1}\le a_{M+1}-\sum_{n=M+1}^{k}b_n
        \le a_{M+1}+(1-v)\eta\sum_{n=1}^{M}r_n$$ holds for all  $k\ge M+1$. Hence, the sequence  $\{a_n\}$ is bounded. By the definition of  $a_n$ in (\ref{an1}), we have
        $$  
        a\|x_{n} - x^*\|^2 + (1-a) \frac{\psi}{\psi-1}\|z_{n+1} - x^*\|^2+\frac{b}{\beta}\|y_{n-1} - y^*\|^2  
        + \frac{1}{\beta}\|\bar{z}_n - y^*\|^2 \le a_n.
        $$
        Hence, the boundedness of  $\{a_n\}$ implies that both sequences $\{x_n\}$, $\{z_n\}$ ,  $\{y_n\}$ and $\{\bar{z}_n\}$ are bounded. Since $ x_{n}^{md} =a x_{n-1} +(1-a)z_{n}$, $ y_{n}^{md} =b y_{n-1} +(1-a)\bar{z}_{n}$, then sequence$\{(x_n,y_n,z_n,\bar{z}_n,x_n^{md},y_n^{md} \}$ is bounded.
        
        (iii)It follows from Lemma \ref{le3} (ii) that the sequence $\{(x_n,y_n)\}$ is bounded. By Assumption \ref{assump:2.2} (ii) and the inequality on the left-hand side of (\ref{t}), we obtain 
        \begin{equation*}
            \|\theta_n\|^2 \leq 2L_{xx}^2\|x_n-x_{n-1}\|^2 +2L_{xy}^2\|y_n-y_{n-1}\|^2 \ \ \text{and}\ \ \Phi_n^y\leq \frac{L_{yy}}{2}\|y_n-y_{n-1}\|^2.
        \end{equation*}
        As a result, the linesearch condition (\ref{lin}) is satisfied provided that
        \begin{align}
           & \nu\omega_1 \delta_{n-1} -\frac{\tau \mu^i}{(1-a)\xi}2\tau_{n-1}L_{xx}^2 > 0
        \ \ \text{and} \ \ \nu b/\beta-\tau \mu^i(2\tau_{n-1}L_{xy}^2+(1-a)\xi L_{yy})/ (1-a)\xi > 0. \label{le33}
        \end{align}
        By $\delta_{n-1}=\tau_{n-1}/\tau_{n-2} \le \varphi$,    $\tau_j\le\tau_{\max}$ and the definition of $\underset {-} {\tau}$ in (\ref{tau-}), it is not difficult to verify that the conditions in (\ref{le33}) are indeed satisfied when $\tau_n \le \underline\tau$.
        
        (iv) We prove it by mathematical induction. We suppose that  $\tau_0\ge\underline\tau$ and $\tau_{\max}\ge \mu \underline\tau $. Assume that $\tau_{n-1}\ge \mu \underline\tau$. To show that the sequence $\{\tau_n\}$ is strictly separated from $0$, we only need to show that $\tau_n\ge \mu \underline\tau$. Since $\tau=\min\{\varphi \tau_{n-1},\tau_{\max}\}$ and  $\varphi>1$, we have $\tau\ge\min\{\tau_{n-1},\tau_{\max}\}\ge \mu \underline\tau$. Recall that $\tau_n=\tau\mu^i$ for some nonnegative integer $i$. If $i=0$, then $\tau_n=\tau\ge \mu \underline\tau$. If $i>0$, then $\hat\tau:=\tau\mu^{i-1}$ must necessarily violate condition (\ref{lin}). It then follows from  Lemma \ref{le3} (ii) that $\hat\tau>\underline\tau$ must hold. Hence, $\tau_n=\mu\hat\tau>\mu\underline\tau >0$.  In addition, it is obvious that $\delta_n=\tau_n /\tau_{n-1}\ge \mu\underline\tau /\tau_{\max}>0$ for all $n$. 
    \end{proof}


    \subsection{Global convergence and sublinear convergence rate.}
        Based on Lemma \ref{le3}, we now establish global  convergence and $O(1/N)$ ergodic sublinear convergence rate of Algorithm \ref{Alg1}.

    \begin{thm}
        \textbf{(Global convergence.)} The sequence $\{(x_n,y_n)\}$ generated by NPDAL-n converges to a solution of the saddle-point problem(\ref{p}).
    \end{thm}

    \begin{proof}
        It follows from Lemma \ref{le2} that $a_{n+1}\le a_n - b_n$. Together with \eqref{le32} and $a_n \ge 0$ for any $n \ge 1$, we obtain
        \begin{align}
             (1-v)(1-\eta)\sum_{n=M+1}^{k}r_n \le a_{M+1}+(1-v)\eta\sum_{n=1}^{M}r_n. \label{T1}
        \end{align}
        From Lemma \ref{le3} (iv), we know $\delta_n\ge \underline{\delta}:=\mu\underline{\tau}/\tau_{\text{max}}>0$ holds for all $n\ge1$. Combining this with the definition of $r_n$, $v\in(0,1)$, $\eta\in[0,1)$ and (\ref{T1}), we deduce that 
        $$\sum_{n=1}^{\infty}r_n<\infty,\ \
        \sum_{n=1}^{\infty}\|x_n-x_{n-1}\|^{2}<\infty,
        \ \ \sum_{n=1}^{\infty}\| y_n-y_{n-1}\|^{2}<\infty,\ \ \text{and}\ \ \sum_{n=1}^{\infty}\|y_n - \bar{z}_n\|^2<\infty.$$ 
        which yield
        $$\lim_{n\to\infty} r_n = \lim_{n\to\infty}\|x_n-x_{n-1}\|=\lim_{n\to\infty}\| y_n-y_{n-1}\| = \lim_{n\to\infty}\|y_n - \bar{z}_n\| = 0.$$ 
        Denote $u_n:=z_{n}-x_{n}$. It is not difficult from $z_{n+1}=\frac{\psi -1}{\psi}x_n+\frac{1}{\psi}z_n$ that
        $$\psi (x_n-x_{n+1})=\psi(z_{n+1}-x_{n+1})-(z_{n}-x_{n})=\psi u_{n+1}-u_n $$ and $$\frac{\psi}{\psi-1}\|u_{n+1}\|^2-\frac{1}{\psi-1}\|u_{n}\|^2+\frac{\psi}{(\psi-1)^2}\|u_{n+1}-u_{n}\|^2=\frac{\psi^2}{(\psi-1)^2}\|x_n-x_{n+1}\|^2,$$
        which implies that $\|u_{n+1}\|^2\le \frac{1}{\psi}\|u_{n}\|^2+\frac{\psi}{\psi-1)}\|x_n-x_{n+1}\|^2$. 
        It then follows from $\sum_{n=1}^{\infty}\|x_{n}-x_{n+1}\|^{2}<\infty$ and Lemma \ref{fact2.3} that $\sum_{n=1}^{\infty}\|u_{n}\|^{2}<\infty$,  which yields that $$\lim_{n\rightarrow\infty}u_{n}=\lim_{n\rightarrow\infty}(z_{n}-x_{n})=0.$$
        
        Combining with the definition of $x_n^{md}$, we obtain  that $\lim_{n\to\infty}(x_n^{md}-x_{n-1})=\lim_{n\to\infty}\frac{1-a}{\psi}(z_{n-1}-x_{n-1})=0$. Similarly, we can get $\lim_{n\to\infty}(\bar{z}_n-y_n)=0$ and $\lim_{n\to\infty}(y_n^{md}-y_{n-1})=0$.
        Since the sequence $\{(x_n,y_n)\}$ is bounded, there exist $(x^{\star},y^{\star})$ and a subsequence $\{n_k:k\ge 1\}\subseteq\{n:n\ge 1\}$ such that $\lim_{k\to\infty}x_{n_k}=x^{\star}$ and $\lim_{k\to\infty}y_{n_k}=y^{\star}.$
         It also follows that
        \begin{align*}
            & \lim_{k\to\infty}x_{n_k+1}
            = \lim_{k\to\infty}x_{n_k}
            =\lim_{k\to\infty}x_{n_k+1}^{md}
            = \lim_{k\to\infty}z_{n_k}
            = x^{\star},\\
            &\lim_{k\to\infty}\bar{z}_{n_k}
            =\lim_{k\to\infty}y_{n_k}^{md}
            =\lim_{k\to\infty}y_{n_k}
            = \lim_{k\to\infty}y_{n_k-1}
            = y^{\star}.
        \end{align*}
        Similar to (\ref{t1}) and (\ref{t3}), for any $(x,y) \in \mathbf{dom}(g) \times \mathbf{dom}(f^{*})$,  the following inequalities hold:
        \begin{equation} \label{T2}
			\begin{split}
                \tau_{n_k}\bigl(g(x_{n_k+1})-g(x)\bigr)&\le \langle x_{n_k+1}-x_{n_k+1}^{md}+\tau_{n_k}\nabla_{\!x}\Phi(x_{n_k},y_{n_k}),\,x-x_{n_k+1}\rangle, \\
                \tau_{n_k}\bigl(f^{*}(y_{n_k})-f^{*}(y)\bigr) &\le \langle y_{n_k}-y_{n_k}^{md}-\tau_{n_k}\nabla_{\!y}\Phi(x_{n_k},y_{n_k-1}),\,y-y_{n_k}\rangle.
            \end{split}
		\end{equation}
        Then, dividing $\tau_{n_k}$ from both sides of (\ref{T2}), using the fact that    both $g$ and $f^{*}$ are closed (and thus lower semicontinuous), and letting $k\to\infty$, we obtain
        \begin{equation}
            g(x^{\star})-g(x)\le\langle\nabla_{\!x}\Phi(x^{\star},y^{\star}),\,x-x^{\star}\rangle
            \quad\text{and}\quad
            f^{*}(y^\star)-f^{*}(y)\le -\langle\nabla_{\!y}\Phi(x^{\star},y^{\star}),\,y-y^{\star}\rangle. \label{T3}
        \end{equation}
        Since (\ref{T3}) holds for any $(x,y)\in \mathbf{dom}(g)\times \mathbf{dom}(f^{*})$, we have
        $$
        -\nabla_{\!x}\Phi(x^{\star}, y^{\star}) \in \partial g(x^{\star})
        \quad\text{and}\quad
        \nabla_{\!y} \Phi(x^{\star}, y^{\star}) \in \partial f^{*},
        $$ 
        which implies that $(x^{\star},y^{\star})$ is a solution of the saddle-point problem (\ref{p}).
         
        Recall that $J(\cdot,\cdot)$ and $a_n$, defined in \eqref{J} and \eqref{an1} respectively, depend on an arbitrarily fixed solution pair $(x^{*},y^{*})$. Since $(x^{\star},y^{\star})$ is also an exact solution to the saddle-point problem \eqref{p}, we can substitute $(x^{*},y^{*})$ with $(x^{\star},y^{\star})$ as the reference point. Consequently, we obtain $\lim_{k\to\infty}a_{n_k}=0$.

        Following the same derivation as in \eqref{le32}, for all $\ell \ge n_k\ge M+1$, we establish
        $$
        \sum_{i=n_k}^{\ell}b_i \ge (1-v)(1-\eta)\sum_{i=n_k}^{\ell}r_i -(1-v)\eta\sum_{i=n_k-M}^{n_k-1}r_i.
        $$
        
        Given that $0\le a_{n+1}\le a_n-b_n$ and $r_n\ge 0$ for all $n\ge 1$, it algebraically follows that
        $$
        0\le a_{\ell+1}\le a_{n_k}+(1-v)\eta\sum_{i=n_k-M}^{n_k-1}r_i, \quad\forall\,\ell\ge n_k\ge M+1.
        $$
        
        Since $\lim_{k\to\infty}a_{n_k}=0$ and $\lim_{n\to\infty}r_n=0$, taking the limit on both sides yields $\lim_{n\to\infty}a_n=0$. Therefore, based on the definitions of $x_n^{md}$, $y_n^{md}$, and $a_n$ in \eqref{an1}, we can rigorously conclude that
        $$
        \lim_{n\to\infty}(x_n,z_n,x_n^{md},y_n,\bar{z}_n,y_n^{md})=(x^{\star},x^{\star},x^{\star},y^{\star},y^{\star},y^{\star}).
        $$
        This completes the proof.\\
    \end{proof}    
        We next establish the ergodic sublinear convergence rate of Algorithm \ref{Alg1} using the primal-dual gap function $J(\cdot,\cdot)$.

    \begin{thm}
        \textbf{(Sublinear convergence rate.)} Suppose the sequence $\{(x_n,y_n)\}$ is generated by Algorithm \ref{Alg1}. Then there exists a constant  $C_1>0$ such that, for any $N\ge 1$,  
        $$
        J(\hat x_N,\hat y_N)\le\frac{C_1}{N},
        $$ 
        where
        $$
        \hat x_N=\frac{1}{s_N}\sum_{n=1}^{N}\tau_nx_n,
        \quad \quad
        \hat y_N=\frac{1}{s_N}\sum_{n=1}^{N}\tau_n y_n,
        \quad \quad
        s_N=\sum_{n=1}^{N}\tau_n.
        $$ 
    \end{thm}
    
    \begin{proof}
        It follows from Lemma \ref{le2} that $2\tau_n J(x_n,y_n)\le a_n-a_{n+1}-b_n$ for all  $n\ge 1$. Summing both sides   from $n=1$ to $N$ and applying  (\ref{le32}), we obtain
        
        $$
        2\sum_{n=1}^{N}\tau_n J(x_n,y_n) \le a_1-a_{N+1}-\sum_{n=1}^N b_n\le a_1+C, $$ 
         where 
        $C=\sum_{n=1}^M |b_n|+(1-v)\eta\sum_{i=1}^{M}r_i.$
       
        Since $J(x,y)$ is jointly convex in $(x,y)$, it follows from the definitions of $\hat x_N$, $\hat y_N$ and Jensen's inequality that
        $$
        J(\hat x_N, \hat y_N)
        \le\frac{1}{s_N}\sum_{n=1}^{N}\tau_n J(x_n,y_n)
        \le\frac{a_1+C}{2s_N}.
        $$ 
        By Lemma \ref{le3}~(iv), we have $\tau_n\ge \mu \underline\tau>0$, and then
        $$
        s_N=\sum_{n=1}^{N}\tau_n\ge \mu \underline\tau N.
        $$ 
       Therefore, with  $C_1=(a_1+C)/(2 \mu \underline{\tau} )>0$, the proof is complete.
    \end{proof}
    
\section{Accelerated Algorithm for special Convex-Concave Saddle-Point Problems}\label{section4}
    Let $g$ and $f$ be the same functions as in (\ref{p}), $H:\operatorname{dom}(g)\rightarrow\operatorname{dom}(f)$ be nonlinear and continuously differentiable, and $h:\mathbb{R}^{q}\rightarrow \mathbb{R}$ be convex and $L_{h}$-smooth with some constant $L_{h}>0$.

  In the section, by adaptively tuning the parameter $\beta$ in Algorithm \ref{Alg1}, we present an accelerated variant  for solving    the following  saddle-point problem  
        \begin{align}
        \min_{x\in \mathbb{R}^{q}}\max_{y\in \mathbb{R}^{p}}\left\{\mathcal{L}(x,y):=g(x)+h(x)+\langle H(x),y\rangle -f^{*}(y)  \right\},  \label{pp1}
    \end{align}
    where $g$ is further  assumed to be  strongly convex  with modulus $\gamma > 0$, i.e.,
    \[
    g(x) \geq g(\tilde{x}) + \langle \zeta, x - \tilde{x} \rangle + \frac{\gamma}{2} \|x - \tilde{x}\|^2,\quad
    \forall x,\tilde{x} \in \mathbf{\mathrm{dom}}(g),\ \zeta \in \partial g(\tilde{x}).
    \]

    Let  
    \begin{equation} \label{phi1}
        \Phi(x,y):=h(x)+\langle H(x),y\rangle. 
    \end{equation}
    Then, formulation \eqref{pp1} is a special case of (\ref{p}), where $\Phi(x,y)$ is linear with respect to the dual variable $y$. By utilizing the Legendre-Fenchel conjugate $f^{**}=f$, 
    the following nonlinear compositional convex optimization problem 
    \begin{align}
        \min_{x\in \mathbb{R}^{q}}\left\{P(x):=g(x)+h(x)+f(H(x))\right\} \label{pp}
    \end{align}
    can be naturally reformulated as problem \eqref{pp1}. Consequently, the proposed accelerated variant of Algorithm~\ref{Alg1} can be directly applied to solve the nonlinear compositional convex optimization problem ~\eqref{pp}.
 
    Assume that Assumptions \ref{assump:2.1}--\ref{assump:2.3} hold for (\ref{pp1}). Moreover, we assume that, for any $y\in \mathrm{dom}(f^{*})$, $\langle H(x),y\rangle$ is convex in $x$. Under this assumption, it follows that $f(H(x))=\max_{y\in \mathbb{R}^{p}}\{\langle H(x),y\rangle-f^{*}(y)\}$ is convex in $x$ as well.
    Since $\Phi(x,y)$ is linear in $y$, the Lipschitz constant $L_{yy}$ defined in Assumption \ref{assump:2.2} (ii) can set to be $0$. Define $l(y):=\max_{x\in \mathbb{R}^{q}}\{\langle H(x),y\rangle-g(x)-h(x)\}$ for $y\in \mathbb{R}^{p}$.

    It follows  from  \eqref{phi1} that $\Phi(x_n, y_n) = h(x_n) + \langle H(x_n), y_n \rangle$, and $\Phi_n^y = h(x_n) + \langle H(x_n), y_{n-1} \rangle + \langle H(x_n), y_n - y_{n-1} \rangle - h(x_n) - \langle H(x_n), y_n \rangle = 0$.

    \subsection{The Proposed Accelerated Algorithm}
    \begin{algorithm}[H] 
    \caption{ aNPDAL-n: Accelerated NPDAL-n for problem \eqref{pp1} when $g$ is strongly convex} 
    \begin{algorithmic} 
        \State \textbf{Initialization:} Choose $\psi\in(1,1+\sqrt{3}),\ (\xi,\xi_1,\varphi,\omega_1) \in \Theta_\psi$, $\tau_{max}>0$, $\mu \in (0,1),\ a\in (0,1)$, and $b\in [0,1)$. Choose $x_0 \in dom_(g),\ y_0 \in dom_(f^*) $, $\beta_0>0$, $\delta_0 > 0$, and $\tau_0 \in (0,\tau_{max}]$. Set $z_0=x_0$ and $\bar{z}_0=y_0$,  and $n=1$.
        \State \textbf{Main Iteration:}
            \State Step 1. Compute
            \begin{gather}
                \kappa_{n-1}=\omega_1 \delta_{n-1}+(1-a)\gamma \tau_{n-1},\quad 
                \rho_n=\frac{\psi-\varphi}{\psi +\gamma\varphi\tau_{n-1}} ,\quad 
                \beta_n=(1+\gamma\rho_n\tau_{n-1})\beta_{n-1}. \label{beta}
            \end{gather}
            \State Step 2. Compute
            \begin{gather}
                z_{n} =\frac{\psi-1}{\psi} x_{n-1} +\frac{1}{\psi} z_{n-1},\quad 
                x_{n}^{md} =a x_{n-1} +(1-a)z_{n}, \label{zx11}\\
                x_{n}=\operatorname{Prox}_{\tau_{n-1} g}\!\left(x_n^{md}-\tau_{n-1} \nabla_x \Phi(x_{n-1},y_{n-1})\right).
            \end{gather}   
            \State Step 3. Compute 
            \begin{gather}
                \bar{z}_n =  (1-b) y_{n-1}+b \bar{z}_{n-1}, \quad
                y_{n}^{md} =b y_{n-1} + (1-b) \bar{z}_n.
            \end{gather}
        \State Step 4. \textbf{Linesearch.} Set $\tau =min\{\varphi\tau_{n-1},\tau_{max} \}$ and   perform 
            \begin{align}        
                y_{n}=\text{Prox}_{\beta_n\tau_n f^*}(y_{n}^{md}+\beta_n\tau_n \nabla_y \Phi(x_{n},y_{n-1})),\label{y1b}
            \end{align}
        \State  \hspace{3.2em}  where   $\tau_{n} :=\tau\mu^i  \;\text{for} \; i=0,1,\cdots $     until the following condition    is satisfied:
            \begin{align}
            &\ \frac{\tau_n \tau_{n-1}}{(1-a)\xi}||\theta_n||^2\leq \frac{\kappa_{n-1}\beta_{n-1}}{\beta_n} \| x_n - x_{n-1}\|^2+\frac{b}{\beta_n}\| y_n - y_{n-1}\|^2 + \frac{1}{\beta_n}(1-b)(1+b)\|y_n - \bar{z}_n\|^2  \label{alin}.
            \end{align}    
        \State Step 5. Set $\delta_n=\tau_n/\tau_{n-1}$, $n\leftarrow n+1$ and go to Step 1.
        \State \textbf{End} 
    \end{algorithmic}  \label{Alg2}
    \end{algorithm} 

    \begin{remark}
        By following the same assumptions as in \cite{ref25}, for $\Phi(x, y) = h(x) + \langle H(x), y \rangle$, we can also prove the convergence and convergence rate of Algorithm \ref{Alg1} using the proof method in Section \ref{section4} of \cite{ref25}, which will not be further elaborated here.
    \end{remark}

    \subsection{Some basic properties of Algorithm \ref{Alg2} }
    \begin{lemma} \label{ale1}
    For  $\theta_n$, $\Phi_y$ and $J(\cdot,\cdot)$ defined in (\ref{tt1}), (\ref{tt2}) and (\ref{J}), respectively, we have
    \begin{equation} \label{at0}
        \begin{split}
            \tau_n J(x_n, y_n) &\leq \langle x_{n+1} - x_{n+1}^{md}, x^* - x_{n+1} \rangle + \delta_n \langle x_n - x_{n}^{md}, x_{n+1} - x_n \rangle+ \tau_n \langle \theta_n,x_n-x_{n+1} \rangle  \\
            &+\tau_n \Phi_n^y+ \frac{1}{\beta_n} \langle y_n - y_{n}^{md}, y^* - y_{n} \rangle -\frac{\gamma\tau_n}{2}(\|x_{n+1}-x_{n}\|^2 +\|x_{n+1}-x^*\|^2). 
        \end{split}
    \end{equation}
    \end{lemma}
    
    \begin{proof}
    Since  $g$ is  strongly convex,  it follows   from (\ref{fact2.2}) that 
    \begin{align}
        \tau_n (g(x_{n+1}) - g(x^*))  &\leq \langle x_{n+1} - x_{n+1}^{md}+ \tau_n \nabla_x \Phi(x_n, y_n), x^* - x_{n+1}\rangle -\frac{\gamma\tau_n}{2}\|x_{n+1}-x^*\|^2, \\
        \tau_{n} (g(x_{n}) - g(x_{n+1}))  &\leq \langle \delta_{n}(x_{n} - x_n^{md}) + \tau_{n} \nabla_x \Phi(x_{n-1}, y_{n-1}), x_{n+1}- x_{n} \rangle -\frac{\gamma\tau_{n}}{2}\|x_{n+1}-x_{n}\|^2.   
    \end{align}
    The proof of the remaining part is entirely analogous to Lemma \ref{le1} and is thus omitted for brevity.
    \end{proof}
    
    \begin{lemma}\label{ale2}
        For all $n \geq 1$, we have
        \begin{align}
            2\tau_n \beta_n J(x_n, y_n) +\beta_{n+1}A_{n+1} \leq \beta_n A_n-\beta_nB_n,  \label{AB}
        \end{align} 
        where 
        \begin{align}
            A_n=&a\|x_{n}- x^*\|^2+\frac{(1-a)\psi+(1-a)\gamma \tau_n}{\psi-1}\| z_{n+1}-x^* \| ^2+(a  +\frac{\kappa_{n-1}\beta_{n-1}}{\beta_n})\|x_n-x_{n-1}\|^2 \nonumber \\
            & +\frac{b}{\beta_n}\|y_{n-1} - y^*\|^2 + \frac{1}{\beta_n}\|\bar{z}_n - y^*\|^2,\label{An1} \\
            B_n=&\frac{b}{\beta_n} \| y_n - y_{n-1}\|^2+ \frac{\kappa_{n-1}\beta_{n-1}}{\beta_n}\|x_n-x_{n-1}\|^2 + \frac{1}{\beta_n}(1-b)(1+b)\|y_n - \bar{z}_n\|^2 \nonumber \\ &- \frac{\tau_n\tau_{n-1}}{(1-a)\xi}\| \theta_n \|^2. \label{Bn1} 
        \end{align}
    \end{lemma}
    
    \begin{proof}
    The proof for the following equation is similar to the one for Lemma \ref{le2}, which yields a result analogous to \eqref{eq14}.
    \begin{equation} \label{p1}
        \begin{split}
            &\quad 2\tau_n J(x_n,y_n) +a\| x_{n+1} - x_{n}\|^2 + a\|x_{n+1} - x^*\|^2 +(1-a)\frac{\psi}{\psi-1}\|z_{n+2}-x^* \|^2 \\
            &+\delta_n \omega_1 \| x_{n+1}- x_{n}\|^2 +\gamma \tau_n(\|x_{n+1}-x_{n}\|^2 +\|x_{n+1}-x^*\|^2) \\
            & +\frac{1}{\beta_n}(1-b)(1+b)\| y_n - \bar z_{n}\|^2 + \frac{1}{\beta_n}b\| y_n - y^*\|^2+\frac{1}{\beta_n}\| \bar{z}_{n+1} - y^*\|^2  \\
            \le & a\|x_{n}- x^*\|^2 +(1-a)\frac{\psi}{\psi-1}\| z_{n+1}-x^* \| ^2  -\frac{1}{\beta_n}b\| y_n - y_{n-1}\|^2  \\
            & + \frac{b}{\beta_n}\|y_{n-1} - y^*\|^2 + \frac{1}{\beta_n}\|\bar{z}_n - y^*\|^2 + \frac{\tau_n\tau_{n-1}}{(1-a)\xi}\| \theta_n \|^2 +  a\| x_{n} - x_{n-1}\|^2,  
        \end{split}
    \end{equation}
    where $\omega_1=(1-a)\omega-a/ \xi_1 >0$, $\Phi_n^y=0$. It follows from \eqref{zx11} that 
    \begin{align}
        &\gamma\tau_n\|x_{n+1}-x^*\|^2=a\gamma\tau_n\|x_{n+1}-x^*\|^2+(1-a)\gamma\tau_n\|x_{n+1}-x^*\|^2  \nonumber \\
        &=a\gamma\tau_n\|x_{n+1}-x^*\|^2+(1-a)\gamma\tau_n \bigg(\frac{\psi}{\psi-1}\| z_{n+2}-x^* \| ^2 -\frac{1}{\psi-1}\| z_{n+1}-x^* \| ^2+\frac{1}{\psi}\| x_{n+1}-z_{n+1} \|^2 \bigg)  \nonumber \\
        &\geq a\gamma\tau_n\|x_{n+1}-x^*\|^2+(1-a)\gamma\tau_n \bigg(\frac{\psi}{\psi-1}\| z_{n+2}-x^* \| ^2 -\frac{1}{\psi-1}\| z_{n+1}-x^* \| ^2 \bigg) . \label{p2}
     \end{align}
    Substituting   (\ref{p2}) into (\ref{p1}) yields 
    \begin{equation} \label{p3}
        \begin{split}
            &\quad 2\tau_n J(x_n,y_n) +a\| x_{n+1} - x_{n}\|^2+a(1+\gamma\tau_n)\|x_{n+1}- x^*\|^2 \\
            & + \frac{(1-a)\psi+(1-a)\psi \gamma \tau_n}{\psi-1}\| z_{n+2}-x^* \|^2 + \delta_n \omega_1 \| x_{n+1}- x_{n}\|^2 +\gamma \tau_n\|x_{n+1}-x_{n}\|^2 \\
            & +\frac{1}{\beta_n}(1-b)(1+b)\| y_n - \bar z_{n}\|^2 + \frac{b}{\beta_n}\| y_n - y^*\|^2+\frac{1}{\beta_n}\| \bar{z}_{n+1} - y^*\|^2  \\
            \le &a\|x_{n} - x_{n-1}\|^2 + a\|x_{n}- x^*\|^2 +\frac{(1-a)\psi+(1-a)\gamma \tau_n}{\psi-1}\| z_{n+1}-x^* \| ^2 \\
            &+ \frac{b}{\beta_n}\|y_{n-1} - y^*\|^2 + \frac{1}{\beta_n}\|\bar{z}_n - y^*\|^2 + \frac{\tau_n\tau_{n-1}}{(1-a)\xi}\| \theta_n \|^2  -\frac{b}{\beta_n}\| y_n - y_{n-1}\|^2.
        \end{split}
    \end{equation}
    From $\tau_{n+1}\le \varphi \tau_n$ and the definition of $\rho_n$ given in \eqref{beta}, we obtain
    \begin{align*}
         \frac{(1-a)(\psi+\psi \gamma \tau_n)}{\psi-1} &= \frac{(1-a)(\psi+ \gamma \tau_{n+1})}{\psi-1} \frac{\psi+\psi \gamma \tau_n}{\psi+ \gamma \tau_{n+1}}  \nonumber \\
        & \geq  \frac{(1-a)(\psi+ \gamma \tau_{n+1})}{\psi-1} \frac{\psi+\psi \gamma \tau_n}{\psi+ \gamma \varphi \tau_{n}}   \nonumber \\
        &= \frac{(1-a)(\psi+ \gamma \tau_{n+1})}{\psi-1}\left(1+\frac{\gamma \tau_n (\psi-\varphi)}{\psi+ \gamma \varphi \tau_{n}}\right) \nonumber \\
        &= \frac{(1-a)(\psi+ \gamma \tau_{n+1})}{\psi-1}\frac{\beta_{n+1}}{\beta_n}.
    \end{align*}
    Since $(\psi-\varphi)/(\psi+\gamma\varphi\tau_n)<1$, we can get $(1+\gamma \tau_n) \geq \left( 1+\frac{\gamma \tau_n (\psi-\varphi)}{\psi + \gamma \varphi \tau_n} \right) = \frac{\beta_{n+1}}{\beta_n}$, i.e., $\beta_n(1+\gamma \tau _n)\geq \beta_{n+1}$.
    
    Using the definitions of $A_n$ and $B_n$, and setting $\kappa_n := \omega_1 \delta_n + (1-a)\gamma \tau_n$, we can complete the proof of this lemma and deduce that (\ref{AB}).
    \end{proof}

    \begin{lemma}\label{ale3}
        The following claims hold.
        \begin{enumerate}
            \item[(i)] The linesearch of Algorithm \ref{Alg2} always terminates;
            \item[(ii)] The sequence $\{(x_n,y_n,z_n,\bar{z}_n)\}$ generated by Algorithm \ref{Alg2} is bounded;
            \item[(iii)] There exist constants $c_1 , \ c_2 >0 $ such that $\sqrt{\beta_n}\tau_n \ge c_1$ and $\beta_n \geq c_2 n^2$ for all $n \geq 1$.
        \end{enumerate}
    \end{lemma}
    
    \begin{proof}
        (i) Assume that the linesearch procedure defined in Algorithm \ref{Alg2} fails to terminate at the $n$-th iteration. Then, in Step 2 of the Algorithm \ref{Alg2}, for all $i=0,1,2,\ldots$ and $\lambda=\tau\mu^{i}$, 
         we have
        \begin{align}
            \frac{\lambda\tau_{n-1}}{(1-a)\xi}\|\theta_{n}(\lambda)\|^{2} & >
            \frac{\kappa_{n-1}\beta_{n-1}}{\beta_{n}}\|x_{n}-x_{n-1}\|^{2}+\frac{b}{\beta_{n}}\|y_{n}(\lambda)-y_{n-1}\|^{2} + \frac{1}{\beta_n}(1-b)(1+b)\|y_n - \bar{z}_n\|^2 \nonumber\\
            & \ge \frac{\kappa_{n-1}\beta_{n-1}}{\beta_{n}}\|x_{n}-x_{n-1}\|^{2}+\frac{b}{\beta_{n}}\|y_{n}(\lambda)-y_{n-1}\|^{2}. \label{ale31}
        \end{align}
        Similar to the proof of Lemma \ref{le3}, for all $\lambda=\tau\mu^{i}$ with $i=0,1,2,\ldots$, we have $y_{n}(\lambda)\in B[y_{n-1};r]$. Then, by combining (\ref{xyL}), (\ref{ale31}) and $\lambda=\tau\mu^{i}$, we obtain
        $\tau \mu^i2\tau_{n-1}L_{xx}^{2}/((1-a)\xi) > \kappa_{n-1}\beta_{n-1}/\beta_{n}$
        or $\tau\mu^{i}(2\tau_{n-1}L_{xy}^{2}/((1-a)\xi)) > b/\beta_{n}$ for all $i\geq 0$. This is impossible since $\mu^{i}\rightarrow 0$ as $i\rightarrow\infty$, which indicates that the linesearch procedure must terminate in finite steps for each iteration.\\
        
        (ii) Since $J(x_{n},y_{n})$, $A_{n}$ and $B_{n}$ are all nonnegative,  and   $\beta_{n}\geq\beta_{n-1}$ for all $n\geq 1$, implied by (\ref{beta}),  we can obtain from (\ref{AB}) that 
        $\beta_1A_n\leq \beta_{n}A_{n}\leq\beta_{n-1}A_{n-1}\leq\cdots\leq\beta_{1}A_{1}.$
        Hence, using the definition of $A_{n}$ in (\ref{An1}), we can derive
        \begin{align*}
            &a\|x_n-x^*\|^2 \le A_1,   \; \frac{(1-a)\psi+(1-a)\gamma \tau _n}{(\psi-1)}\|z_{n+1}-x^{*}\|^{2}\leq A_{1}, \quad \nonumber \\
            & \|y_{n-1}-y^{*}\|^{2}\leq\beta_{n}A_{n}\leq\beta_{1}A_{1}, \quad \|\bar{z}_n-y^{*}\|^{2}\leq\beta_{n}A_{n}\leq\beta_{1}A_{1},
        \end{align*}
        which implies that $\{(x_n,z_{n},y_n,\bar{z}_n):n\geq 1\}$  is bounded.\\
        
        (iii) Define  $h(\tau) := 1 + \frac{(\psi-\varphi)\gamma \tau}{\psi + \gamma\varphi\tau}$, which is   increasing  in $\tau>0$. Since $\tau_{n}\leq\tau_{\max}$ for all $n\geq 0$, it follows that $\varsigma:=h(\tau_{\max})\geq h(\tau_{n})=1+\gamma\rho_{n+1}\tau_{n}\geq 1$. It follows from part (ii) of this lemma that the sequence $\{(x_{n},y_{n}):n\geq 1\}$ is bounded. Moreover, using the left-hand-side inequality in (\ref{xyL}), we know that the linesearch condition (\ref{alin}) is satisfied provided that
        \begin{align}
            &2\tau_n\tau_{n-1}L_{xx}^{2}/((1-a)\xi) \le \kappa_{n-1}\beta_{n-1}/\beta_{n},    
             \quad 2\tau_n\tau_{n-1}L_{xy}^{2}/((1-a)\xi) \le b/\beta_{n},  \label{ale32}
        \end{align}
        Since $\kappa_{n-1}=\omega_1\delta_{n-1}+(1-a)\gamma \tau_{n-1} \ge \omega_1\delta_{n-1}$, $\beta_{n}=(1+\gamma\rho_{n}\tau_{n-1})\beta_{n-1}\leq\varsigma\beta_{n-1}$ and $\tau_{n}\leq\tau_{\max}$ for all $n\geq 1$, the inequalities in (\ref{ale32}) hold if
        \begin{align}
            \frac{\kappa_{n-1}\beta_{n-1}}{\tau_{n-1}\sqrt{\beta_{n}}}&\geq\frac{\omega_1\beta_{n-1}}{\tau_{n-2}\sqrt{\beta_{n}}}\geq\frac{\omega\sqrt{\beta_{0}}}{\tau_{\max}\sqrt{\varsigma}}. \label{ale33}
        \end{align}
        Denote $\varUpsilon_{1}:=\frac{(1-a)\xi}{2L_{xx}^{2}}\frac{\omega_1\sqrt{\beta_{0}}}{\tau_{\max}\sqrt{\varsigma}}$ and $\varUpsilon_{2}:=\frac{(1-a)\xi  b}{2L_{xy}^{2}}$. 
        It follows from (\ref{ale33}) that the left-hand-side inequality in (\ref{ale32}) holds when $\sqrt{\beta_{n}}\tau_{n}\leq\varUpsilon_{1}$, while the right-hand-side inequality is equivalent to $\beta_{n}\tau_{n}\tau_{n-1}\leq\varUpsilon_{2}$. 
        To proceed, for any fixed $n\ge 1$, we introduce an auxiliary threshold 
        $$
        \hat{\tau}_n := \min\left\{\frac{\varUpsilon_1}{\sqrt{\beta_n}}, \frac{\varUpsilon_2}{\beta_n \tau_{n-1}}\right\}.
        $$
        A direct substitution verifies that if the stepsize were set to $\hat{\tau}_n$, then both inequalities in (\ref{ale32}) and hence the linesearch condition (\ref{alin}) would be satisfied. Recalling that in Algorithm 2 the actual stepsize $\tau_n$ is chosen as the first element of the sequence $\{\tau \mu^i\}$ satisfying (\ref{alin}), and that the linesearch condition is monotone with respect to decreasing stepsizes, it follows that $\tau_n$ must be strictly larger than $\mu \hat{\tau}_n$. Otherwise, the preceding trial $\tau_n/\mu \le \hat{\tau}_n$ would already satisfy the condition, contradicting the minimality of $i$. Consequently, for all $n\ge 1$, we must have either 
        $$
        \tau_n > \frac{\mu \varUpsilon_1}{\sqrt{\beta_n}} \quad \text{or} \quad \tau_n > \frac{\mu \varUpsilon_2}{\beta_n \tau_{n-1}}.
        $$
        In the former case, we have $\sqrt{\beta_{n}}\tau_{n}>\mu\varUpsilon_{1}$. In the latter case, with $n$ replaced by $n+1$, we have $\beta_{n+1}\tau_{n}\tau_{n+1}>\mu\varUpsilon_{2}$. Then, by $\tau_{n+1}\leq\varphi\tau_{n}$, we derive $\beta_{n+1}\tau_{n}^{2}\geq\beta_{n+1}\tau_{n}\tau_{n+1}/\varphi>\mu\varUpsilon_{2}/\varphi$, and thus
        \begin{center}
             $\beta_{n}\tau_{n}^{2}=\frac{\beta_{n+1}\tau_{n}^{2}}{1+\gamma\rho_{n+1}\tau_{n}}>\frac{\mu\varUpsilon_{2}/\varphi}{1+\gamma\rho_{n+1}\tau_{n}}\geq\frac{\mu\varUpsilon_{2}}{\varphi\varsigma}>0
            $.
        \end{center}
        Hence, we always have     $\sqrt{\beta_{n}}\tau_{n}>c_{1}:=\min(\mu\varUpsilon_{1},\sqrt{\mu\varUpsilon_{2}/(\varphi\varsigma)})>0$. Since $\tau_{n}\leq\tau_{\max}$, we have
        \begin{align}
            \beta_{n+1}=\beta_{n}\left(1+\frac{(\psi-\varphi)\gamma \tau_n}{\psi +\gamma\varphi\tau_n}\right)\geq\beta_{n}+\frac{(\psi-\varphi)\gamma\sqrt{\beta_{n}}\tau_{n}}{\psi +\gamma\varphi\tau_{\max}}\sqrt{\beta_{n}}\geq\beta_{n}+\varrho\sqrt{\beta_{n}}, \label{ale34}
        \end{align}
        where $\varrho:=\frac{(\psi-\varphi)\gamma c_1}{\psi+\gamma\varphi\tau_{max}}>0$. From (\ref{ale34}), it is easy to show by induction that $\beta_{n}\geq c_{2}n^{2}$ for all $n\geq 1$ with $c_{2}:=\min(\varrho^{2}/9,\beta_{1})>0$. This completes the proof.
    \end{proof}

    \subsection{Convergence and Convergence Rate}
        \begin{thm}[Convergence results] \label{con2}
            Let $\{(z_n, x_n, y_n, \beta_n, \tau_n): n \geq 1\}$ be the sequence generated by Algorithm \ref{Alg2}. Then there exist constant $C_1 > 0$ such that,  for every $N \ge 1$,
            $$\|x_{N+1} - x^*\| \leq C_1 / (N + 1), \quad    J(\hat{x}_N, \hat{y}_N) \leq \frac{\beta_1}{c_3} \frac{A_1}{N^2},$$ 
            where $\hat{x}_N$ and $\hat{y}_N$ are given by
            \begin{align}
                \hat{x}_N = \frac{1}{s_N} \sum_{n=1}^N \beta_n \tau_n x_n \ \ \text{and} \ \ \hat{y}_N = \frac{1}{s_N} \sum_{n=1}^N \beta_n \tau_n y_n \ \ \text{with}\ \ s_N = \sum_{n=1}^N \beta_n \tau_n.\label{c0}
            \end{align}
        \end{thm}

    \begin{proof}
        (i) It follows from (\ref{alin}) and (\ref{Bn1}) that $B_n \geq 0$ for any $n \geq 1$. Let $y \in \mathbb{R}^p$ be arbitrarily fixed. By dropping $B_n$ on the right-hand side of (\ref{AB}) and taking a sum over $n = 1, \ldots, N$, we obtain
        \begin{align}
              \beta_{N+1} A_{N+1} + 2\sum_{n=1}^N \beta_n \tau_n J(x_n, y_n) \leq \beta_1 A_1.  \label{c1}
        \end{align}
        Together with the fact that $J(x_n,  y_n)$ is always nonnegative,   we infer from (\ref{c1}) and the definition of $A_n$ in (\ref{An1}) that
        \begin{align}
            a\|x_{N+1} - x^*\|^2 \leq \frac{ \beta_1 A_1}{\beta_{N+1}}\le \frac{\beta_1A_1}{c_2(N+1)^2}.  \label{c2}
        \end{align}
        Therefore, it follows from (\ref{c2}) and item (iii) of Lemma \ref{ale3}  that $\|x_{N+1} - x^*\| \leq C_1 / (N + 1)$ with $C_1 := \sqrt{ \beta_1 A_1 / ac_2} > 0$.  
    
        (ii) $J(x,  y)$ is jointly convex in $(x,y )$. Dropping  the nonnegative term $\beta_{N+1} A_{N+1}$ from (\ref{c1}), and applying   (\ref{c0}) with Jensen's inequality, we obtain   
        \begin{align}
             2J(\hat{x}_N, \hat{y}_N) \leq \frac{2}{s_N} \sum_{n=1}^N \beta_n \tau_n J(x_n, y_n) \leq \frac{\beta_1 A_1}{s_N}.  \label{c3}
        \end{align} 
        Since $\varphi > 1$, we   have $\rho_{n+1} < \frac{\psi-\varphi}{\psi} = 1 - \frac{\varphi}{\psi} < 1$. It follows from (\ref{beta}) that $\beta_n\tau_n = \frac{\beta_{n+1}-\beta_n}{\rho_{n+1}\gamma} \ge \frac{\beta_{n+1}-\beta_n}{\gamma}$.
        Summing from $n=1$ to $N$   and using Lemma~\ref{ale3}(iii) with $\beta_{N+1} \geq c_2 (N + 1)^2$, we obtain  $$s_N = \sum_{n=1}^N \beta_n \tau_n \geq (\beta_{N+1} - \beta_1) / \gamma \geq c_3 N^2$$ for some $c_3 > 0$. Therefore, it follows from (\ref{c3}) that
        $J(\hat{x}_N, \hat{y}_N) \leq \frac{\beta_1A_1}{2c_3} \frac{1}{N^2}.$ This completes the proof.
    \end{proof}

\section{A New Fully Adaptive Proximal Gradient Method}\label{section5}
    In the section, we present an adaptive variant of  Algorithm \ref{Alg1} to solve the saddle-point problem \eqref{finite_sum_saddle}, namely, \begin{equation}\label{finite_sum_saddle2}
        \min_{x\in \mathbb{R}^{q}} \max_{y\in \mathbb{R}^{p}} \; g(x) + \langle H(x), y \rangle - \iota_{\{\mathbf{1}/p\}}(y).
    \end{equation}
    where $H(x) := (h_1(x), \dots, h_p(x))^\top$, $\mathbf{1} \in \mathbb{R}^p$ denotes the vector of all ones, and $\iota_{\{\mathbf{1}/p\}}$ is the indicator function of the singleton set $\{\mathbf{1}/p\}$. 
    It is obvious that   this problem can be   cast as a special case of problem \eqref{p} by assigning $\Phi(x,y) = \langle H(x), y \rangle$ and $f^*(y) = \iota_{\{\mathbf{1}/p\}}(y)$. As mentioned in Section \ref{section1},    the composite convex optimization problem  \eqref{finite_sum}, namely,     \begin{equation}\label{finite_sum2}
        \min_{x\in \mathbb{R}^{q}} \; g(x) + \frac{1}{p}\sum_{i=1}^{p} h_i(x).
    \end{equation}
    can be transformed into problem \eqref{finite_sum_saddle2}.  Consequently, the proposed adaptive variant of Algorithm~\ref{Alg1} can be directly applied to solve  the composite convex optimization problem\eqref{finite_sum2}.    
 
    We next   exploit  the specific structure of problem \eqref{finite_sum_saddle2} to propose  an adaptive variant of  Algorithm \ref{Alg1}.
 
    Applying Algorithm \ref{Alg1} to solve the saddle-point problem \eqref{finite_sum_saddle2}, we can readily obtain the following:

    (i)The primal gradient takes the form $\nabla_x \Phi(x,y) = H'(x)^{\top} y$, where the Jacobian is given by $H'(x) = [\nabla h_1(x), \ldots, \nabla h_p(x)]^{\top}$;

    (ii)$\nabla_y \Phi(x,y) = H(x)$ and the local Lipschitz constant $L_{yy}$ reduces to $0$;

    (iii)The dual variable remains constant at $y_n \equiv \mathbf{1}/p \in \mathbb{R}^p$ for all $n \geq 1$. 

    It also follows that $\nabla_x \Phi(x_n, y_n) = \nabla h(x_n)$ for any $n \geq 1$, and thus $\theta_n = \nabla h(x_n) - \nabla h(x_{n-1})$ and $\Phi_n^y \equiv 0$ for all such $n$.

    By setting $\beta = 1$, and using $\Phi_n^y = 0$ and $y_n \equiv \mathbf{1}/p $,  the linesearch condition \eqref{lin} correspondingly simplifies correspondingly 
    \begin{equation} \label{eq:linesearch_mod}
        \frac{\tau_n \tau_{n-1}}{(1-a)\xi} \|\theta_n\|^2 \leq \nu r_n + (1-\nu)c_n,
    \end{equation}
    which together with $\delta_{n-1} = \tau_{n-1}/\tau_{n-2}$ implies that   
    $$
    \tau_n \leq \frac{(1-a)\xi \bigl( \nu r_n + (1-\nu) c_n \bigr)}{\delta_{n-1}\tau_{n-2}\|\theta_n\|^2}.
    $$
    By setting $\eta = 0$, we immediately have $c_n = 0$ for all $n \ge 1$, which, together with $r_n = \omega_1 \|x_n - x_{n-1}\|^2$ and $\theta_n = \nabla h(x_n) - \nabla h(x_{n-1})$, implies that
    $$
    \tau_n \leq \frac{(1-a)\nu\xi\omega_1}{\delta_{n-1}\tau_{n-2}} \cdot \frac{\| x_n - x_{n-1} \|^2}{\| \nabla h(x_n) - \nabla h(x_{n-1}) \|^2}.
    $$

    Based on the above analysis, we obtain an adaptive variant of Algorithm \ref{Alg1}, which is indeed a linesearch-free proximal gradient method.
    
    \begin{algorithm}[H] 
    \caption{NPGM: New adaptive proximal gradient method with two convex combinations  }
        \begin{algorithmic} 
            \State \textbf{Initialization:} Choose $\psi\in(1,1+\sqrt{3}),\ (\xi,\xi_1,\varphi,\omega_1) \in \Theta_\psi$, $\tau_{max}>0$, $\nu \in (0,1),\ \mu \in (0,1),\ a\in (0,1)$. Choose $x_0 \in \textbf{dom}(g)$ and $\tau_0 \in (0,\tau_{max}]$. Set $z_0=x_0$ and $n=1$.
            \State \textbf{Main Iteration:}
                \State Step 1. Compute 
                \begin{gather}
                    z_{n} =\frac{\psi-1}{\psi} x_{n-1} +\frac{1}{\psi} z_{n-1}, \quad
                    x_{n}^{md} =a x_{n-1} +(1-a)z_{n},  \\
                    x_{n}=\text{Prox}_{\tau_{n-1}}(x_n^{md}-\tau_{n-1} \nabla_x \Phi(x_{n-1}, \mathbf{1}/p)).
                \end{gather}
            \State Step 2. Set
                $$\tau_n = \min \Bigl\{ \varphi \tau_{n-1}, \;
                \frac{\nu \xi (1-a)\omega_1}{\tau_{n-2}} \cdot \frac{\| x_n - x_{n-1} \|^2}{\| \nabla h(x_n) - \nabla h(x_{n-1}) \|^2}, \;\tau_{\max} \Bigr\}.$$
            \State Step 3. $n \leftarrow n + 1$ and go to Step 1.
        \end{algorithmic}  \label{Alg3}
    \end{algorithm} 
    
\section{Numerical Experiments}\label{section6}
    In this section, we evaluate the numerical performance of the proposed algorithms through comprehensive numerical experiments, benchmarking it against several state-of-the-art primal-dual methods on quadratically constrained quadratic programming (QCQP) problems and sparse logistic regression (SLR) problems. All computational tests are executed using \textsc{Matlab} R2023b on a 64-bit Windows   equipped with an Intel Core i5-10500 CPU (3.10 GHz) and 8 GB of RAM. 
    
    The subsequent evaluations are structured into two primary experimental groups.
    
    \subsection{Adaptive Step Size Ratio \texorpdfstring{$\beta$}{beta}}
    
    To ensure a fair and meaningful comparison across all tested algorithms, we employ a dynamic update strategy for the step size ratio $\beta$. This mechanism is designed to continuously balance the primal and dual infeasibilities, which denoted as $\operatorname{pinf}_n$ and $\operatorname{dinf}_n$, respectively. Following the established practice in \cite{ref25}, by applying Moreau's decomposition to the dual update step, we can implicitly define an auxiliary variable $w_n := \operatorname{Prox}_{f/(\beta\tau_n)}(y_n^{md}/(\beta\tau_n) + H(x_n))$, which yields the relation $y_n = y_n^{md} + \beta\tau_n(H(x_n) - w_n)$. Based on this, we dynamically adjust the step size ratio $\beta$ by balancing the primal and dual infeasibilities, denoted as $\operatorname{pinf}_n$ and $\operatorname{dinf}_n$ respectively, which are computed as:
    \begin{equation}
        \operatorname{pinf}_n := \|H(x_n) - w_n\|_1 = \frac{1}{\beta\tau_n} \|y_n - y_n^{md}\|_1,
    \end{equation}
    and
    \begin{equation}
        \operatorname{dinf}_n := \frac{\operatorname{dist}(-H'(x_n)^\top y_n - \nabla h(x_n), \partial g(x_n))}{1 + \|x_n\|_1},
    \end{equation}
    where $\operatorname{dist}(v, S)$ represents the distance from the vector $v$ to the set $S$ measured by the $\ell_1$-norm.
    Following the established practice in \cite{ref25}, $\beta$ is adjusted via the following adaptive rule:
    \begin{equation}\label{beta1}
        \beta =
        \begin{cases}
            \max\{0.8\beta, \underline{\beta}\}, & \text{if } r_n \leq 0.8, \\
            \beta, & \text{if } r_n \in (0.8, 1.25), \\
            \min\{1.25\beta, \overline{\beta}\}, & \text{if } r_n \geq 1.25,
        \end{cases} 
    \end{equation}
    where $r_n =\frac{\operatorname{pinf}_n}{\operatorname{dinf}_n}$ signifies the residual ratio, and $\underline{\beta}$ and $\overline{\beta}$ impose strict lower and upper bounds to prevent numerical instability. Consistent with the configuration of PDAc-L \cite{ref25}, we initialize $\beta = 1$ for all   experiments and enforce the bounding interval as $\underline{\beta} = 0.01$ and $\overline{\beta} = 100$. Furthermore, the    parameters specific to the baseline PDAc-L algorithm are strictly preserved as recommended in its original literature \cite{ref25}.

    \subsection{Quadratically Constrained Quadratic Programming Problems}
         In the subsection, we benchmark the empirical efficiency of the proposed \textcolor{black}{NPDAL-n} (Algorithm \ref{Alg1}) against two established baselines: the primal-dual algorithm with backtracking (PDB) \cite{ref10} and PDAc-L \cite{ref25}. The evaluation is conducted on a suite of convex quadratically constrained quadratic programming (QCQP) problems using synthetically generated datasets, adopting the experimental protocol from \cite{ref10}.

        The QCQP instances are formulated as 
        \begin{equation}\label{QCQP}
        \begin{aligned}
            h_{\mathrm{opt}} := \min_{x \in X} & \quad h(x) := \frac{1}{2}x^{\top}A_0 x + b_0^{\top}x \\
            \text{s.t.} & \quad h_j(x) := \frac{1}{2}x^{\top}A_j x + b_j^{\top}x - c_j \le 0, \quad j \in \{1, \dots, m\},
        \end{aligned}
        \end{equation}
        where the feasible bounding box is defined as $X := [-10, 10]^n$. To construct the problem instances, the linear coefficient vectors $\{b_j\}_{j=0}^m \subset \mathbb{R}^n$ are sampled from a standard normal distribution, while the scalar offsets $\{c_j\}_{j=1}^m \subset \mathbb{R}$ are drawn uniformly from $[0,1]$. The symmetric positive semi-definite matrices are constructed via eigen-decomposition $A_j = \Lambda_j^{\top} S_j \Lambda_j$ for all $j \in \{0, 1, \dots, m\}$. Here, each $\Lambda_j \in \mathbb{R}^{n \times n}$ represents a randomly generated orthogonal matrix, and $S_j \in \mathbb{R}_+^{n \times n}$ is a diagonal matrix containing eigenvalues independently and uniformly distributed across $[0,100]$ (where zero eigenvalues are explicitly permitted).
        
        With the notation $H(x) := (h_1(x), \dots, h_m(x))^{\top}$ and  $\Phi(x,y) := h(x) + \langle y, H(x) \rangle$, the QCQP model can be naturally reformulated as the following minimax problem:
        $$
        \min_x \max_y \; g(x) + \Phi(x,y) - f^*(y),
        $$
        where $g(x) := \iota_X(x)$ restricts the primal variable to the box $X$, and $f^*(y) := \iota_{\mathbb{R}_+^m}(y)$ imposes non-negativity on the dual multipliers. Since $\Phi(x,y)$ exhibits strict linearity with respect to $y$, the partial gradients evaluate directly to $\nabla_x \Phi(x,y) = H'(x)^{\top}y + A_0x + b_0$ and $\nabla_y \Phi(x,y) = H(x)$, with $H'(x)$ denoting the Jacobian. Crucially, under this structural property, the dual linearization error vanishes identically (i.e., $\Phi_n^y \equiv 0$). Consequently, verifying the linesearch condition requires negligible computational overhead for both PDAc-L and \textcolor{black}{NPDAL-n}.
        
        During execution, the step size ratio $\beta$ dynamically adapts via the mechanism in \eqref{beta1}. We obtain the   optimal value $h_{\mathrm{opt}}$ utilizing the MOSEK solver (interfaced through CVX\footnote{Downloaded from \url{http://cvxr.com/cvx/}}). The baseline PDB is terminated upon satisfying $\max\{e_{\mathrm{obj}}(x_n), e_{\mathrm{con}}(x_n)\} \le \epsilon$, where the relative objective error and maximal constraint violation are computed as 
        \begin{equation}\label{e}
            e_{\mathrm{obj}}(x) := \frac{|h(x) - h_{\mathrm{opt}}|}{|h_{\mathrm{opt}}|} \quad \text{and} \quad e_{\mathrm{con}}(x) := \frac{1}{m}\sum_{j=1}^{m}\max\{h_j(x), 0\}.
        \end{equation}
        For PDAc-L and \textcolor{black}{NPDAL-n}, the algorithmic terminations  include both the aforementioned error tolerance and an additional primal-dual stationarity condition: $\max\{\operatorname{pinf}_n, \operatorname{dinf}_n\} < \epsilon_{pd}$. A hard safeguard of $n_{\mathrm{max}}$ iterations is uniformly applied across all methods to prevent infinite loops. For this specific experiment, the tolerances are tightly fixed at $\epsilon = 10^{-8}$ and $\epsilon_{pd} = 10^{-6}$, capped at a maximum of $n_{\mathrm{max}} = 5 \times 10^4$ iterations.

        In the QCQP experiment, the parameter settings of each algorithm are as follows:
        \par NPDAL-n: $a = 0.02,\ b = 0.775,\ \psi = 1.9,\ \varphi = 1.3,\ \xi = 0.69,\ \xi_1 = 0.8,\ \mu = 0.7,\ \nu = 0.9,\ \ M = 5,\ \eta = 0.9,\ \omega = 2\psi - \xi - (\psi^3 \varphi) / (1 + \psi),\ \omega_1 = (1 - a)\omega - (a / \xi_1).$
        \par PDAc-L: $\psi = 2,\ \phi = 6/5,\ \nu = 0.9,\ \mu = 0.7,\ \xi = 2/5,\ M = 5,\ \eta = 0.9$ \cite{ref25}
        \par PDB: $\eta = 0.7,\ \tau_0 = 0.001,\ \gamma_0 = 1,\ \mu = 0$ \cite{ref10}.
        
        Table \ref{table1} reports  the computational   results of PDB, PDAc-L, and \textcolor{black}{NPDAL-n} across various problem dimensions, where ``Iter'' denotes the number of iterations required for convergence, ``Time'' denotes the CPU time in seconds, and ``\#LS'' denotes the  number of steps triggered by the linesearch subroutine.

        \begin{table}[htbp]
            \centering
            \caption{Comparison of PDB, PDAc-L and NPDAL-n on the QCQP problems with different values of (n,m).}
            \label{table1}
            \begin{tabular}{|c|c|c|c|c|c|c|c|c|c|c|}
                \hline
                \multirow{2}{*}{n} & \multirow{2}{*}{m} & \multicolumn{3}{c|}{NPDAL-n} & \multicolumn{3}{c|}{PDAc-L} & \multicolumn{3}{c|}{PDB} \\ \cline{3-11}
                & & Iter & Time & \#LS & Iter & Time & \#LS & Iter & Time & \#LS \\ \hline
                100 & 10   & 192  & 0.0855  & 88   & 355  & 0.1415  & 171  & 2345 & 1.2213   & 2282 \\ \hline
                100 & 30   & 245  & 0.2018  & 114  & 1051 & 0.5359  & 542  & 6872 & 10.4429  & 6703 \\ \hline
                100 & 50   & 370  & 0.5728  & 260  & 1546 & 1.0059  & 789  & 8855 & 33.1045  & 8640 \\ \hline
                500 & 10   & 200  & 0.7257  & 141  & 710  & 2.1863  & 358  & 2491 & 24.9635  & 2428 \\ \hline
                500 & 30   & 382  & 3.5985  & 365  & 952  & 8.0479  & 485  & 4753 & 137.7556 & 4637 \\ \hline
                500 & 50   & 845  & 12.6326 & 609  & 2077 & 28.9895 & 1050 & 8164 & 383.6959 & 7969 \\ \hline
                1000& 10   & 227  & 2.7184  & 162  & 678  & 7.1844  & 340  & 2408 & 93.6593  & 2348 \\ \hline
                1000& 30   & 648  & 20.9743 & 467  & 1150 & 33.9674 & 579  & 5936 & 673.3687 & 5793 \\ \hline
                1000& 50   & 993  & 57.1765 & 720  & 1783 & 84.4744 & 902  & 5997 & 1088.3978& 5854 \\ \hline
            \end{tabular}
        \end{table}

        As demonstrated by Table \ref{table1}, \textcolor{black}{NPDAL-n} outperforms both PDAc-L and PDB across all dimensional scales. Specifically, \textcolor{black}{NPDAL-n} requires significantly fewer iterations and lower linesearch overhead, which collectively translate to a substantial reduction in total CPU time. Notably, as the variable dimension $n$ and the number of constraints $m$ scale up concurrently, the baseline PDB algorithm suffers from severe   degradation in efficiency. In contrast, \textcolor{black}{NPDAL-n} demonstrates highly stable and graceful scalability. Its relative acceleration over the benchmark solvers becomes increasingly pronounced under complex, high-dimensional, and heavily constrained conditions.

        To further visualize the convergence behavior, we use the QCQP problem with  $n=500$  and $m=10$ to evaluate the three test algorithms across 10 independent random instances. Figure \ref{Figure-2} plots the evolution of the relative objective error $e_{\mathrm{obj}}(x_n)$ and the maximal  constraint violation $e_{\mathrm{con}}(x_n)$ against both the iteration number and CPU time. In these visualizations, the bold solid lines trace the median convergence trajectories, while the surrounding shaded envelopes capture the dispersion range observed across the randomized trials, thereby highlighting the robust empirical reliability of our approach.

        As depicted in Figure \ref{Figure-2},    NPDAL-n   exhibits pronounced advantages over the baseline methods in terms of both iteration complexity and overall computational efficiency. Specifically, NPDAL-n demonstrates a significantly faster decay rate in both the relative objective error $e_{\mathrm{obj}}(x_n)$ and constraint violation $e_{\mathrm{con}}(x_n)$, requiring substantially fewer iterations to reach high-precision solutions compared to PDB and PDAc-L. Furthermore, the CPU time trajectories reveal a crucial computational insight: the overhead introduced by our adaptive linesearch mechanism is practically negligible, as it is vastly outweighed by the accelerated convergence it provides. Finally, the narrower shaded envelopes associated with NPDAL-n indicate a tighter variance across independent random trials, thereby corroborating the enhanced robustness and stability of the proposed double convex combination strategy when dealing with complex, non-linear coupling structures.
        \begin{figure}[H]
            \centering           
                \begin{minipage}{0.45\textwidth}
                    \centering
                    \includegraphics[width=\textwidth]{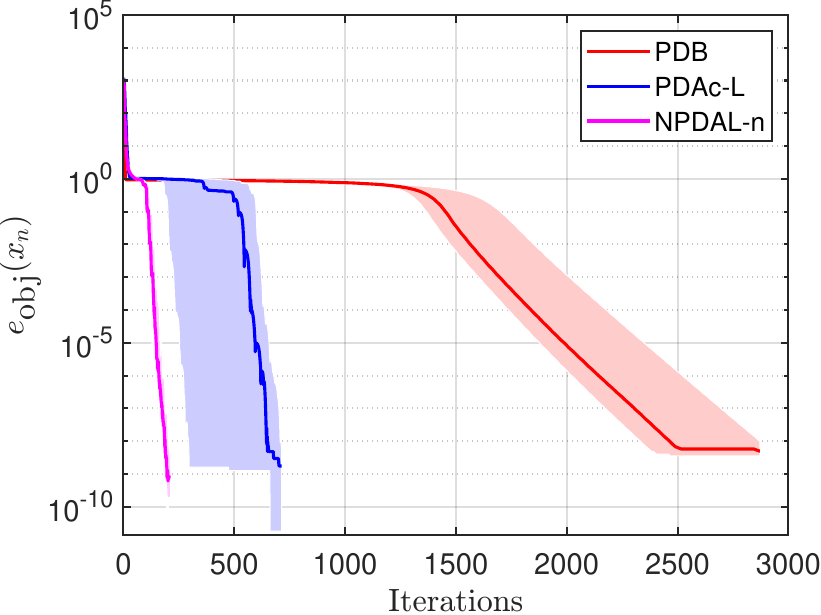}
                \end{minipage}
            \hspace{0.05\textwidth}
                \begin{minipage}{0.45\textwidth}
                \centering
                \includegraphics[width=\textwidth]{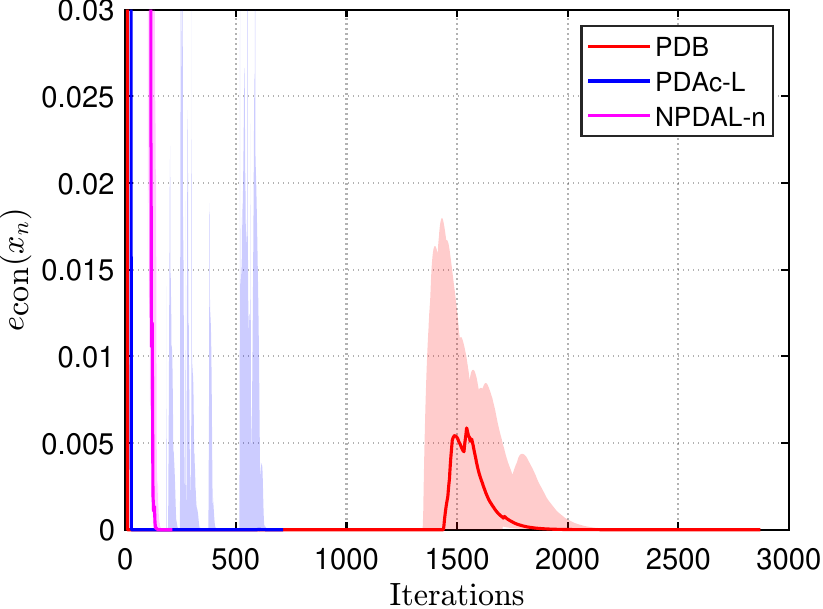}         
                \end{minipage}        
            \vspace{0.02\textheight} 
                \begin{minipage}{0.45\textwidth}
                \centering
                \includegraphics[width=\textwidth]{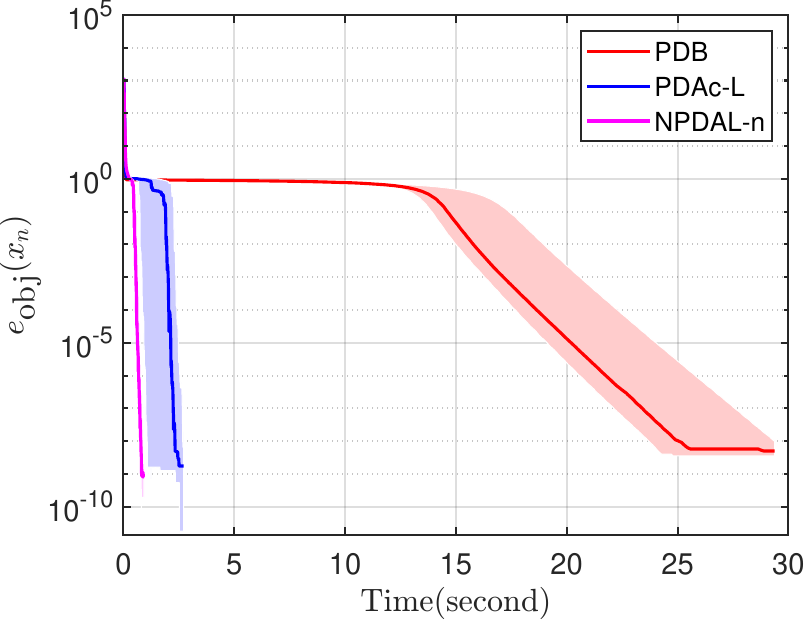}         
                \end{minipage}
            \hspace{0.05\textwidth}
                \begin{minipage}{0.45\textwidth}
                \centering
                \includegraphics[width=\textwidth]{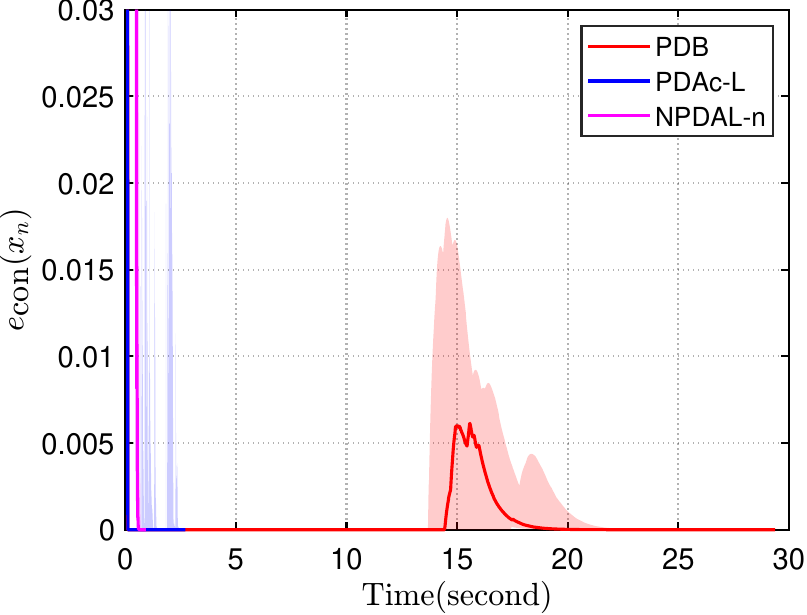}    
                \end{minipage}
            \caption{Convergence behavior of the PDB, PDAc-L, and  NPDAL-n   on randomly generated QCQP instances with $n=500$ and $m=10$. The evolution of the relative objective error $e_{\mathrm{obj}}(x_n)$ and constraint violation $e_{\mathrm{con}}(x_n)$ are plotted against the iteration count and CPU time (in seconds). Solid lines indicate the median trajectories across 10 independent random trials, while the shaded envelopes represent the range of variation.}\label{Figure-2}
        \end{figure}
    
        To assess the performance of the accelerated algorithm aNPDAL-n (Algorithm \ref{Alg2}) in solving problem (\ref{QCQP}) with a strongly convex $h$, we conducted experiments using data generated in the same way as for the convex case. The way to generate the dataset is the same as that in \cite{ref25}: For $j = 0$, we initialize the diagonal elements of $S_{0}$ randomly from $[1, 101]$ and decompose it as $S_{0} = \tilde{S}_{0} + I$, where $\tilde{S}_{0} \succeq 0$. This yields $A_{0} = \Lambda_{0}^{\top}(\tilde{S}_{0} + I)\Lambda_{0} = \tilde{A}_{0} + I$, which allows us to rewrite $h(x)$ as $\hbar(x) + \frac{1}{2}\|x\|^{2}$, with $\hbar(x) := \frac{1}{2}x^{\top}\tilde{A}_{0}x + b_{0}^{\top}x$.
        Consequently, we can apply aPDAc-L to solve the min-max problem:
        $$ 
        \min_{x}\max_{y}\; g(x) + \Phi(x, y) - \iota_{+}(y), 
        $$
        where $g(x) := \iota_{X}(x) + \frac{1}{2}\|x\|^{2}$ is strongly convex, $\Phi(x, y) := \hbar(x) + \langle H(x), y \rangle$, and $f^{*}(y) := \iota_{+}(y)$.
        
        In this experiment, we benchmark the standard \textcolor{black}{NPDAL-n} against its accelerated counterpart (\textcolor{black}{aNPDAL-n}) and the accelerated baseline aPDAc-L \cite{ref25} on the above strongly convex QCQP instances. In addition, we also present numerical results of PDAc‑L and NPDAL‑n for solving strongly convex QCQP problems to demonstrate the effectiveness of the accelerated variants. All three methods are initialized with an identical step size ratio of $\beta_0 = 1$. The evaluation is conducted over a set of 10 independently generated strongly convex instances with  $n = 500$ and $m = 10$. Figure \ref{aQCQP} reveals the empirical benefits of the acceleration schemes. Notably, \textcolor{black}{aNPDAL-n} maintains a certain computational edge over aPDAc-L in terms of the iterations and  CPU  time. More importantly, when benchmarked against the unaccelerated \textcolor{black}{NPDAL-n}, \textcolor{black}{aNPDAL-n} achieves significant performance improvements, reaching convergence with drastically fewer iterations and much shorter runtime.

        For supplementary illustration, we list the parameter choices for each algorithm in this numerical experiment:
        \par PDAc-L: $\psi = 2,\ \phi = 1.2,\ \nu = 0.9,\ \mu = 0.7,\ \xi = 2/5,\ M = 5,\ \eta = 0.9$ \cite{ref25}.
        \par NPDAL-n: $\psi = 2,\ \varphi = 1.1,\ \xi = 0.3,\ \xi_1 = 0.9,\ a = 0.3,\ b = 0.5,\ \mu = 0.75,\ \nu = 0.9,\ \ M = 5,\ \eta = 0.9,\ \omega = 2\psi - \xi - (\psi^3 \varphi) / (1 + \psi),\ \omega_1 = (1 - a)\omega - (a / \xi_1).$
        \par aPDB: $\eta = 0.7,\ \tau_0 = 0.003,\ \gamma_0 = 1,\ \mu = sc$ \cite{ref10}.
        \par aPDAc-L: $\psi = 2,\ \phi = 6/5,\ \nu = 0.9,\ \mu = 0.7,\ \xi = 2/5$ \cite{ref25}.
        \par aNPDAL-n: $\psi = 2,\ \varphi = 1.1,\ \xi = 0.2,\ \xi_1 = 0.9,\ a = 0.3,\ b = 0.7,\ \mu = 0.8,\ \omega = 2\psi - \xi - (\psi^3 \varphi) / (1 + \psi),\ \omega_1 = (1 - a)\omega - (a / \xi_1)$.
        
        where the ``sc'' in aPBD is the strong convexity parameter of the function $g = \iota_{X}(x) + \frac{1}{2}\|x\|^{2}$. 
        \begin{figure}[H]
            \centering          
            \begin{minipage}{0.45\textwidth}
                \centering
                    \includegraphics[width=\textwidth]{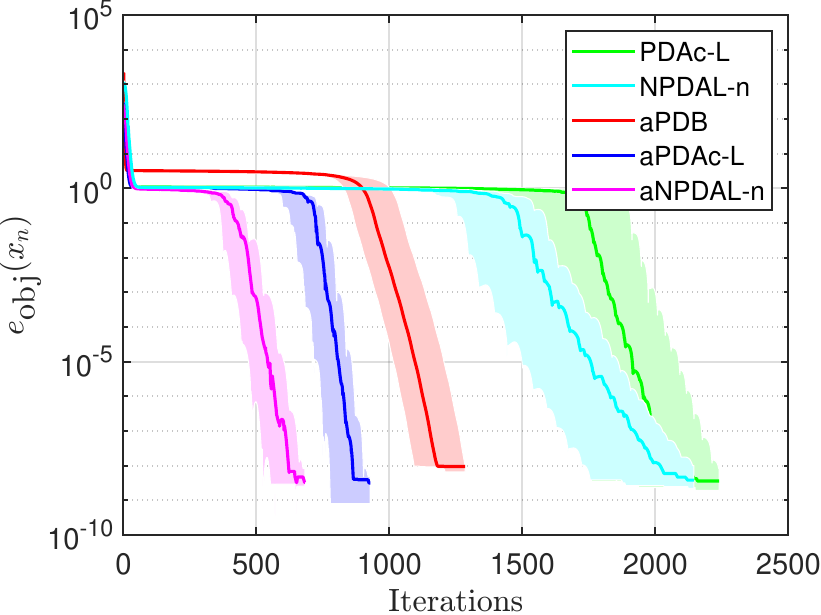}
            \end{minipage}
            \hspace{0.05\textwidth}
                \begin{minipage}{0.45\textwidth}
                \centering
                \includegraphics[width=\textwidth]{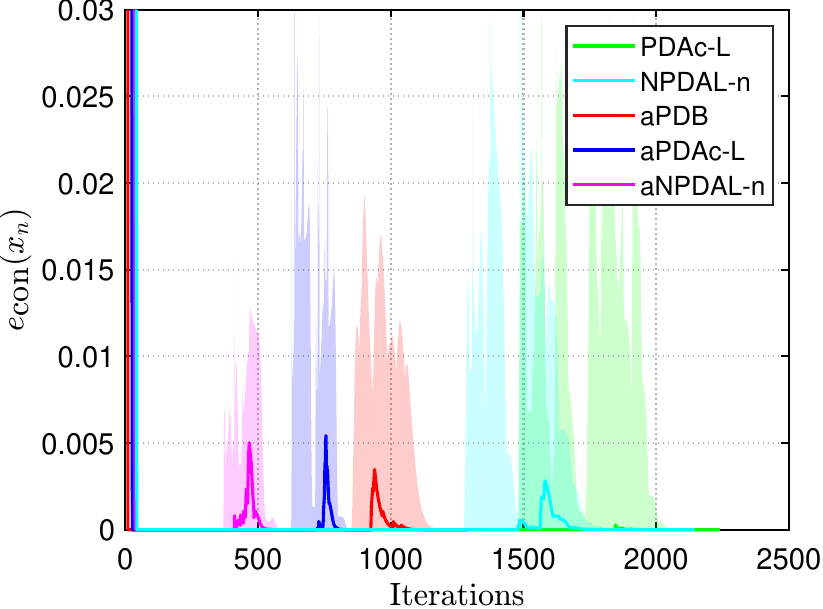}         
                \end{minipage}        
            \vspace{0.01\textheight} 
                \begin{minipage}{0.45\textwidth}
                \centering
                \includegraphics[width=\textwidth]{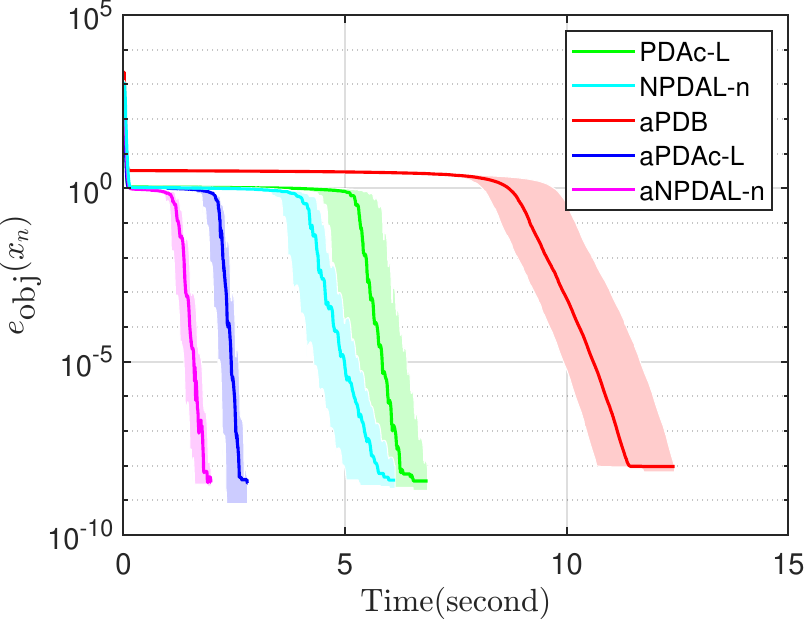}         
                \end{minipage}
            \hspace{0.05\textwidth}
                \begin{minipage}{0.45\textwidth}
                \centering
                \includegraphics[width=\textwidth]{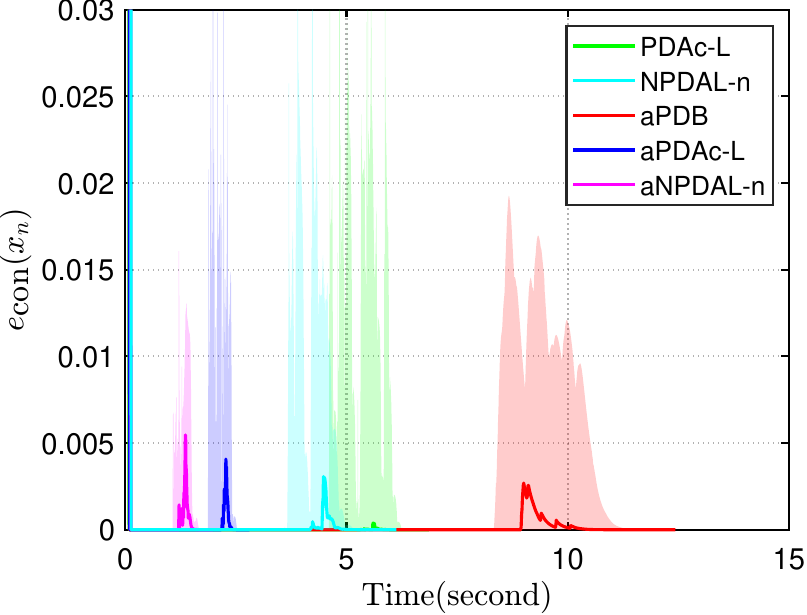}    
                \end{minipage}
            \caption{Convergence comparison of the standard NPDAL-n, the proposed accelerated variant (aNPDAL-n), and the accelerated baseline aPDAc-L on strongly convex QCQP instances ($n=500$, $m=10$). All algorithms are initialized with $\beta_0 = 1$. The trajectories explicitly illustrate the substantial performance gains achieved by the acceleration scheme in reducing both iteration complexity and CPU time (in seconds). Solid lines represent the median results across 10 independent random trials, with shaded envelopes indicating the variation range.}\label{aQCQP}
        \end{figure}
        
    \subsection{Sparse Logistic Regression Problems}
      To   evaluate the performance of NPGM (Algorithm \ref{Alg3}),  we investigate sparse logistic regression (SLR) for binary classification.  Given a training dataset comprising $m$ samples, denoted as $\{(a_{i}, b_{i}) \in \mathbb{R}^{n} \times \{\pm 1\} \mid i=1,\ldots,m\}$, where $a_i$ represents the feature vector and $b_i$ is the binary label, the SLR optimization model is formulated as:
        \begin{equation} \label{SLR}
            F_{\mathrm{opt}} := \min_{x \in \mathbb{R}^{n}} \left\{ F(x) := t\|x\|_{1} + \frac{1}{m}\sum_{i=1}^{m}\log\bigl(1+\exp(-b_{i}a_{i}^{\top}x)\bigr) \right\},
        \end{equation}
        where $t > 0$ serves as the sparsity-inducing regularization parameter. By partitioning the objective into a non-smooth penalty $g(x) := t\|x\|_{1}$ and a smooth empirical loss $h(x) := \frac{1}{m}\sum_{i=1}^{m}\log\bigl(1+\exp(-b_{i}a_{i}^{\top}x)\bigr)$, the SLR model \eqref{SLR} perfectly matches the composite structure of \eqref{finite_sum}. 
        
        Consistent with the experimental design in \cite{ref25}, the regularization parameter is configured as $t = 0.005\|A^{\top}b\|_{\infty}$, where $A^{\top} = [a_{1}, a_{2}, \dots, a_{m}]$ acts as the feature matrix and $b = (b_{1}, b_{2}, \dots, b_{m})^{\top}$ is the label vector. We benchmark the performance of  NPGM against two state-of-the-art competitors: aPGMc \cite{ref25} and aGRAAL \cite{ref14}. The algorithmic parameters for the baselines are rigorously set to their recommended optimal values: $\psi = 3/2$ and $\varphi = 10/9$ for aGRAAL, $\psi = 2$ and $\varphi = 6/5$ for aPGMc, and $(\psi,\ \varphi) = (2.7,\ 4/3)$ and $(\psi,\ \varphi) = (2.55,\ 3/2)$ for NPGM.
        
        The empirical evaluations are conducted on two widely-adopted real-world datasets sourced from the LIBSVM\footnote{Website: \url{https://www.csie.ntu.edu.tw/~cjlin/libsvmtools/datasets/}} repository: \texttt{a9a} (with $m=32,561, n=123$) and \texttt{rcv1} (with $m=20,242, n=47,236$). Adhering to the termination protocol in \cite{ref14}, all test  algorithms run until the stationarity criterion based on the proximal gradient mapping
        $$
        \|x_{n} - \operatorname{Prox}_{\tau_{n}g}\bigl(x_{n} - \tau_{n}\nabla h(x_{n})\bigr)\| \le 10^{-6}.
        $$
         is met.
        The surrogate for the global minimum, $F_{\mathrm{opt}}$, is defined as the lowest objective value $F(x_n)$ attained across all iterations. Figure \ref{SLRfig} compares the decay of the objective residual  $F(x_n) - F_{\mathrm{opt}}$   versus the iteration count and CPU time.

        \begin{figure}[H]
            \centering          
            \begin{minipage}{0.45\textwidth}
                \centering
                \includegraphics[width=\textwidth]{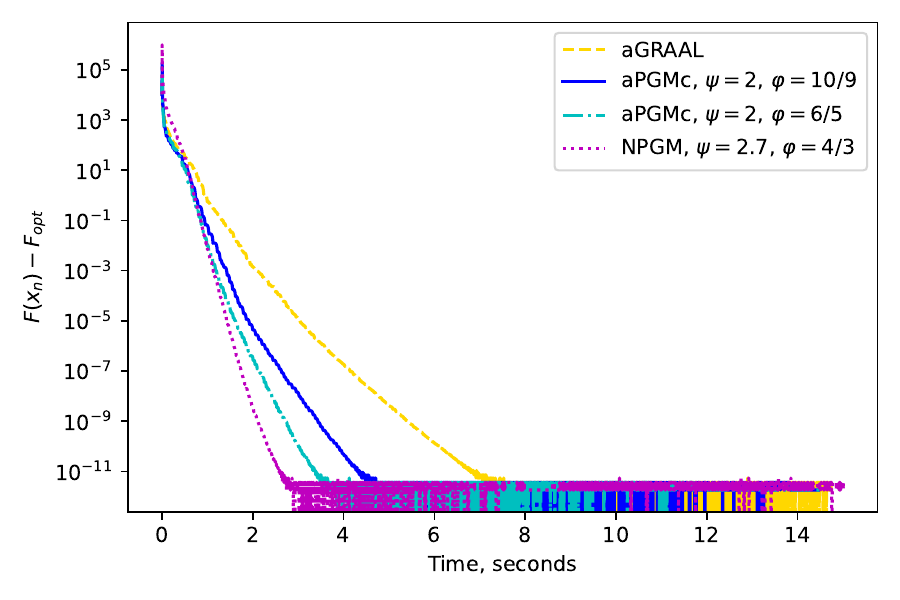}
                \subcaption{\texttt{a9a}}
            \end{minipage}
            \hspace{0.05\textwidth}
            \begin{minipage}{0.45\textwidth}
                \centering
                \includegraphics[width=\textwidth]{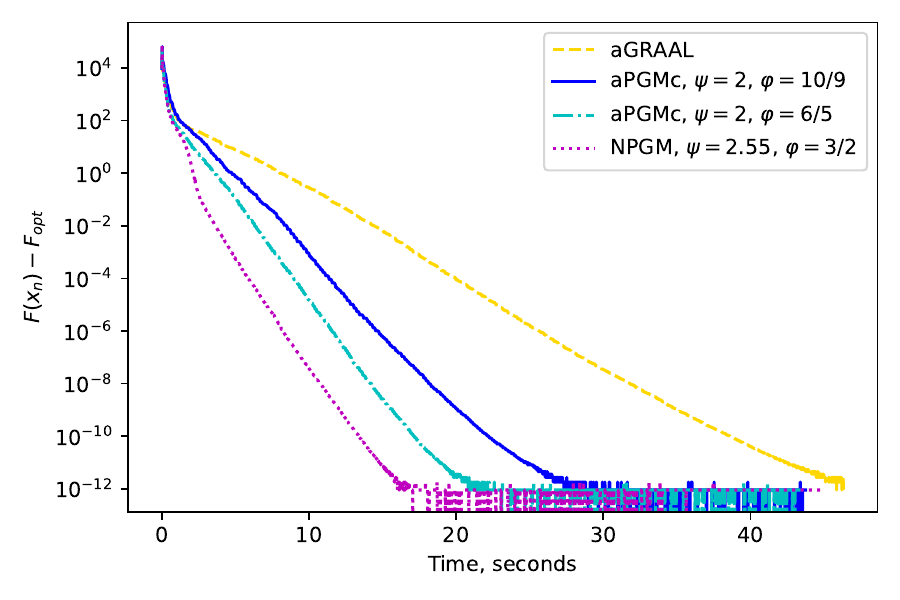} 
                \subcaption{\texttt{rcv1}}        
            \end{minipage}
            \caption{Performance comparison of aPGMc, aGRAAL, and the proposed adaptive method on the sparse logistic regression (SLR) problem. The decay of the objective function gap $F(x_n) - F_{\mathrm{opt}}$ is plotted with respect to CPU time. The evaluation is conducted on two real-world datasets from LIBSVM: \texttt{a9a} (left) and \texttt{rcv1} (right).}\label{SLRfig}
        \end{figure}

\section{Conclusion}\label{section7}

    We develop a new primal-dual algorithm with two convex combinations and linesearch, denoted as NPDAL-n, for convex-concave saddle-point problems. The proposed two convex combinations in NPDAL-n ensure that   the permissible range of the convex combination parameters is mainly determined   by theoretical considerations, with little regard for numerical performance. Theoretically, we have rigorously established the global pointwise convergence of the iterative sequence and derived its sublinear ergodic convergence rate under standard assumptions.
    
    Furthermore, we present an accelerated variant  of NPDAL-n when the primal function is strongly convex and $\Phi(x,y):=h(x)+\langle H(x),y\rangle$. Additionally, by isolating the primal sequence from the primal-dual updates, we demonstrate  that our framework naturally reduces to a linesearch-free proximal gradient method  for a composite convex optimization problem. The numerical experiments on quadratically constrained quadratic programming and sparse logistic regression problems demonstrate that the proposed algorithms consistently outperform state-of-the-art benchmarks—including PDB, PDAc-L, and aGRAAL. Notably, our methods exhibit superior scalability, reduced linesearch overhead, and enhanced robustness, particularly when tackling high-dimensional and tightly constrained instances.\\

    \noindent\textbf{Declaration of competing interest}\\
        The authors declare that they have no known competing financial interests or personal relationships that could have appeared to influence the work reported in this paper.\\
    
    \noindent\textbf{Acknowledgements}\\ 
        This research is supported by the National Science Foundation of China (Nos. 12261019,12571329).\\

    \noindent\textbf{Data availability}\\
        The datasets generated during and analyzed during the current study are available from the corresponding author on reasonable request.\\

    \noindent\textbf{CRediT authorship contribution statement}\\
        \textbf{Shuning Liu:} Writing - original draft, Investigation, Formal analysis, Conceptualization, Visualization. 
        \textbf{Zexian Liu:} Writing - review \& editing, Conceptualization, Software, Methodology, Funding acquisition, Resources, Supervision, Validation. 
        \textbf{Jialong Li:} Writing - review \& editing, Formal analysis, Visualization, Software, Validation.

  
\end{document}